\documentclass[preprint,11pt]{elsarticle}
\usepackage[T1]{fontenc}
\usepackage{fullpage}
\usepackage{amsmath,amssymb,amsthm,bm,mathtools,mathrsfs}
\usepackage{lineno}
\usepackage{enumitem}
\setlist[enumerate]{label=\textup{(\roman*)},leftmargin=*}
\usepackage{xcolor}
\usepackage{hyperref}
\hypersetup{
	colorlinks=true,
	linkcolor=blue!55!black,
	urlcolor=blue!55!black,
	pdftitle={A weighted inequality for matroid flats and log-concavity of Whitney numbers},
	pdfauthor={Yuxian Dong, Yanxin Liu, Jianxi Mao},
	pdfkeywords={matroids, Whitney numbers, log-concavity, Chow rings}
}

\newtheorem{thm}{Theorem}[section]

\newtheorem{coro}[thm]{Corollary}

\newtheorem{conj}[thm]{Conjecture}

\newtheorem{theorem}[thm]{Theorem}
\newtheorem{lemma}[thm]{Lemma}
\newtheorem{corollary}[thm]{Corollary}
\newtheorem{proposition}[thm]{Proposition}

\theoremstyle{definition}

\newtheorem{definition}[thm]{Definition}
\theoremstyle{remark}

\newtheorem{remark}[thm]{Remark}
\numberwithin{equation}{section}
\newcommand{\R}{\mathbb R}
\newcommand{\C}{\mathbb C}
\newcommand{\Z}{\mathbb Z}
\newcommand{\F}{\mathcal F}

\newcommand{\rBC}{\overline{\operatorname{BC}}}
\newcommand{\rk}{\operatorname{rk}}

\newcommand{\Span}{\operatorname{span}}
\newcommand{\Ann}{\operatorname{Ann}}
\newcommand{\diag}{\operatorname{diag}}

\newcommand{\degr}{\operatorname{deg}}

\makeatletter

\newcommand{\Rmnum}[1]{\expandafter\@slowromancap\romannumeral #1@}
\makeatother
\allowdisplaybreaks[4]
\begin{document}
	\begin{frontmatter}
		\title{A weighted inequality for matroid flats and log-concavity of Whitney numbers}
		\author{Yuxian Dong}
		\ead{yuxiand@hotmail.com}
		
		\author{Yanxin Liu\corref{cor1}}
		\ead{liu-yanxin@outlook.com}
		\cortext[cor1]{Corresponding author}
		
		\author{Jianxi Mao}
		\ead{maojx@dlut.edu.cn}
		
		\address{School of Mathematical Sciences, Dalian University of Technology, Dalian 116024, P.R. China}
		\begin{abstract}
			Using the Hodge--Riemann relations for matroid Chow rings,
			we establish a weighted inequality for flats of a loopless matroid.
			As applications, we prove a sharp form of the
			points--lines--planes conjecture and Brylawski's conjecture
			on the ultra-log-concavity of Whitney numbers of the first kind.
			We characterize the equality cases in both results.
			We also recover the strong Mason conjecture on the ultra-log-concavity
			of the numbers of independent sets of each size.
		\end{abstract}
	\end{frontmatter}
	
	\section{Introduction}\label{sec:introduction}
	A \emph{matroid} $M$ on a finite ground set $E$ is specified by a \emph{rank function} $\rk_M:2^E\to\Z_{\ge0}$ satisfying the following properties:
	\begin{enumerate}
		\item[{\rm (i)}]
		$0\le \rk_M(X)\le |X|$ for every $X\subseteq E$;
		
		\item[{\rm (ii)}]
		$\rk_M(X)\le \rk_M(Y)$ whenever $X\subseteq Y\subseteq E$;
		
		\item[{\rm (iii)}]
		$\rk_M(X\cup Y)+\rk_M(X\cap Y)\le \rk_M(X)+\rk_M(Y)$
		for all $X,Y\subseteq E$.
	\end{enumerate}
	We write $r=\rk_M(E)$ for the rank of $M$.
	A subset $F\subseteq E$ is a \emph{flat} of $M$ if $\rk_M(F\cup\{e\})>\rk_M(F)$ for every $e\in E\setminus F$.
	The flats of $M$, ordered by inclusion, form a geometric lattice $\mathcal L(M)$, called the \emph{lattice of flats}; see~\cite[Theorem~1.7.5]{Oxley11}.
	For $0\le i\le r$, let
	\begin{equation*}
		\F_i(M)=\{F\in\mathcal L(M):\rk_M(F)=i\}.
	\end{equation*}
	The \emph{Whitney numbers of the second kind} of $M$ are $W_i(M)=|\F_i(M)|$ for $0\le i\le r$.
	We refer the reader to~\cite{Oxley11,Welsh76} for standard matroid terminology and facts.
	
	Rota~\cite{Rota71} conjectured that the sequence $(W_i(M))_{i=0}^r$ is unimodal:
	\begin{equation*}
		W_0(M)\le\cdots\le W_m(M)
		\ge\cdots\ge W_r(M)
		\quad\text{for some }0\le m\le r.
	\end{equation*}
	Mason~\cite{Mason72} proposed the stronger conjecture that this sequence is log-concave:
	\begin{equation}\label{masonlc}
		W_i(M)^2\ge W_{i-1}(M)W_{i+1}(M)
		\quad \text{for all } 1\le i\le r-1.
	\end{equation}
	Since $W_i(M)>0$ for $0\le i\le r$, log-concavity implies unimodality; see~\cite[Ch.~15]{Welsh76}.
	
	Divoux, Larson, Lowen, and Wang~\cite{DLLW26}
	gave a counterexample to the unimodality of $(W_i(M))_{i=0}^r$, thereby disproving both conjectures.
	The case $i=2$ of Mason's conjecture~\eqref{masonlc} remained open.
	The flats of ranks $1$, $2$, and $3$ are called \emph{points}, \emph{lines}, and \emph{planes}, respectively.
	Mason~\cite[Conjecture~1E]{Mason72} proposed the following strengthening of log-concavity at $i=2$, known as the \emph{points--lines--planes conjecture}; see also~\cite{Oxley11,Welsh76}.
	
	\begin{conj}
		\label{conj:plp}
		Let $M$ be a matroid of rank at least $3$.
		Then
		\begin{equation*}
			W_2(M)^2\ge\frac32\, W_1(M)\,W_3(M).
		\end{equation*}
	\end{conj}
	
	Mason~\cite[Conjecture~1D]{Mason72} also proposed the following strengthening of Conjecture~\ref{conj:plp}; see also~\cite{Duk04,Seymour82}.
	
	\begin{conj}\label{conj:splp}
		Let $M$ be a matroid of rank at least $3$.
		Then
		\begin{equation}\label{eq:splp}
			W_2(M)^2
			\ge
			\frac{3\,(W_1(M)-1)}{2\,(W_1(M)-2)}
			\,W_1(M)\,W_3(M),
		\end{equation}
		with equality if and only if every plane contains exactly three points.
	\end{conj}
	
	Stonesifer~\cite{Stonesifer75} proved Conjecture~\ref{conj:splp} for graphic matroids.
	Seymour~\cite{Seymour82} extended it to matroids in which every line contains at most four points.
	For further results and related work, see~\cite{Duk04,Duk08,Kun00,MMIB12,Pur86}.
	
	An element $e\in E$ is a \emph{loop} if $\rk_M(\{e\})=0$.
	The matroid $M$ is \emph{loopless} if it has no loops, or equivalently, if $\varnothing$ is a flat.
	For a loopless matroid, the least and greatest elements of $\mathcal L(M)$ are $\varnothing$ and $E$, respectively.
	
	Let $\mu$ be the \emph{M\"obius function} of $\mathcal L(M)$.
	For flats $F\subseteq G$, it is determined recursively by
	\begin{equation*}
		\mu(F,G)=
		\begin{cases}
			1,
			& F=G,\\[2mm]
			-\displaystyle\sum_{\substack{H\in\mathcal{L}(M)\\
					F\subseteq H\subsetneq G}}
			\mu(F,H),
			& F\subsetneq G.
		\end{cases}
	\end{equation*}
	For a loopless matroid $M$, the~\emph{characteristic polynomial} is
	\begin{equation*}
		\chi_M(t)
		=
		\sum_{F\in\mathcal L(M)}
		\mu(\varnothing,F)t^{r-\rk_M(F)}
		=
		\sum_{i=0}^{r}(-1)^i w_i(M)t^{r-i}.
	\end{equation*}
	The coefficients $w_i(M)$ are called the \emph{unsigned Whitney numbers of the first kind}.
	Rota~\cite{Rota71} and Heron~\cite{Her72} conjectured that the sequence $(w_i(M))_{i=0}^r$ is unimodal.
	Welsh~\cite{Welsh76} later proposed the stronger conjecture that this sequence is log-concave.
	Huh~\cite{Huh12} proved the conjecture for matroids representable over a field of characteristic zero, and Huh and Katz~\cite{HK12} extended the result to all representable matroids.
	Adiprasito, Huh, and Katz~\cite{AHK18} proved the conjecture for arbitrary matroids by establishing the hard Lefschetz theorem and the Hodge--Riemann relations for matroid Chow rings; see also~\cite{ADH23}.
	In 1982, Brylawski~\cite[p.~232]{Brylawski82} proposed the following strengthening.
	
	\begin{conj}
		\label{conj:brylawski}
		Let $M$ be a loopless matroid of rank $r$ on $n$ elements.
		Then the sequence $(w_i(M))_{i=0}^r$ is \emph{ultra-log-concave of order $n$}:
		\begin{equation*}
			\frac{w_i(M)^2}{\binom{n}{i}^2}\ge \frac{w_{i-1}(M)}{\binom{n}{i-1}} \frac{w_{i+1}(M)}{\binom{n}{i+1}}
			\quad \text{for all } 1\le i\le r-1.
		\end{equation*}
	\end{conj}
	
	Equivalently, Conjecture~\ref{conj:brylawski} asserts that
	\begin{equation*}
		w_i(M)^2
		\ge
		\frac{(i+1)(n-i+1)}{i(n-i)}
		\,w_{i-1}(M)\,w_{i+1}(M)
		\quad \text{for all }1\le i<r.
	\end{equation*}
	Brylawski~\cite{Brylawski82} verified the conjecture for truncations of Boolean algebras, partition lattices, and modular geometries.
	Stanley's factorization for supersolvable lattices~\cite[Theorem~4.1]{Sta72}, together with Newton's inequalities, gives ultra-log-concavity of order $r$.
	
	We prove Conjectures~\ref{conj:splp} and~\ref{conj:brylawski} by a common weighted inequality.
	To state it, let $M$ be a loopless matroid of rank $r\ge1$.
	Since $\chi_M(1)=0$, we define its \emph{reduced characteristic polynomial} by
	\begin{equation*}
		\overline{\chi}_M(t)
		=\frac{\chi_M(t)}{t-1}
		=\sum_{i=0}^{r-1}(-1)^i\,\nu_i(M)\,t^{r-1-i}.
	\end{equation*}
	The broken-circuit theorem~\cite{Brylawski77} gives $\nu_i(M)>0$ for $0\le i\le r-1$; see also~\cite[p.~443]{AHK18}.
	Comparing coefficients in
	$\chi_M(t)=(t-1)\overline{\chi}_M(t)$, we obtain
	\begin{equation}\label{eq:w-nu}
		w_i(M)=\nu_i(M)+\nu_{i-1}(M)
		\quad \text{for all }0\le i\le r,
	\end{equation}
	where we set $\nu_{-1}(M)=\nu_r(M)=0$.
	In particular, $\nu_0(M)=w_0(M)=1$.
	
	For $S\subseteq E$, the \emph{restriction} $M|S$ is the matroid on $S$ with rank function $\rk_{M|S}(X)=\rk_M(X)$ for $X\subseteq S$.
	For a flat $F$, the \emph{contraction} $M/F$ is the matroid on $E\setminus F$ with rank function
	\[
	\rk_{M/F}(S)=\rk_M(S\cup F)-\rk_M(F)
	\quad\text{for all }S\subseteq E\setminus F.
	\]
	Our main result is the following weighted inequality.
	
	\begin{thm}\label{thm:weighted-incidence}
		Let $M$ be a loopless matroid of rank $r$, and let $1\le i\le r-1$.
		Assign a real weight $a_F$ to each flat $F\in\F_i(M)$.
		Then
		\begin{equation}\label{ineq:weight}
			\sum_{H\in\F_{i+1}(M)}
			\frac{\left(\sum_{F\in \F_{i}(M|H)}a_F\right)^2}
			{\nu_{i-1}(M|H)}
			\le
			\frac{\left(\sum_{F\in\F_i(M)}a_F\right)^2}
			{\nu_{i-1}(M)}
			+
			\sum_{F\in\F_i(M)}
			\frac{|\F_1(M/F)|-1}{\nu_{i-1}(M|F)}
			\,a_F^2.
		\end{equation}
	\end{thm}
	
	Taking $a_F=1$ for all $F\in\F_2(M)$ in~\eqref{ineq:weight},
	we obtain the following result.
	
	\begin{thm}\label{thm:splp}
		Conjecture~\ref{conj:splp} holds. Equality holds if and only if every plane contains
		exactly three points.
	\end{thm}
	The points--lines--planes conjecture follows immediately.
	\begin{coro}\label{thm:plp}
		Conjecture~\ref{conj:plp} holds.
	\end{coro}
	
	Taking $a_F=|\mu(\varnothing,F)|$ for all $F\in\F_i(M)$ in~\eqref{ineq:weight},
	we obtain the following result.
	
	\begin{thm}\label{thm:brylawski}
		Conjecture~\ref{conj:brylawski} holds.
		For each $1\le i\le r-1$, equality holds if and only if every circuit
		of $M$ has at least $i+3$ elements.
	\end{thm}
	
	The rest of the paper is organized as follows.
	Section~\ref{sec:geometry} gives a geometric proof of Theorem~\ref{thm:weighted-incidence} when the relevant truncation is realizable over $\C$.
	In Section~\ref{sec:weighted}, we prove the theorem for arbitrary loopless matroids using their Chow rings.
	Section~\ref{sec:app} gives the applications to the points--lines--planes and Brylawski conjectures, including the equality cases, and recovers the strong Mason conjecture.
	
	\section{A geometric proof of the weighted inequality}\label{sec:geometry}
	
	In this section, we prove the weighted inequality when the relevant truncation of $M$ is realizable over $\C$.
	The proof uses the intersection form on a surface in a wonderful model.
	
	Let $M$ be a loopless matroid of rank $r\ge2$ on $E$.
	For flats $F\subseteq G$, define the \emph{interval}
	\[
	[F,G]=\{H\in\mathcal L(M)\mid F\subseteq H\subseteq G\}.
	\]
	For a flat $F$, the lattice of flats of $M|F$ is $[\varnothing,F]$, and $\mathcal{L}(M/F)\cong[F,E]$.
	If $\varnothing\subsetneq F\subsetneq E$, then both $M|F$ and $M/F$ are loopless and have positive rank.
	
	For $1\le t\le r$, the \emph{rank $t$ truncation} $T_t(M)$ is the matroid on $E$ with rank function
	\begin{equation}\label{truncation}
		\rk_{T_t(M)}(S)=\min\{\rk_M(S),t\} \quad\text{for all }S\subseteq E.
	\end{equation}
	The lattice of flats of $T_t(M)$ is
	\begin{equation*}
		\mathcal{L}(T_t(M))=\left(\,\bigcup_{i=0}^{t-1}\F_i(M)\right)\cup\{E\},
	\end{equation*}
	where $E$ has rank $t$.
	Truncation preserves the lower intervals $[\varnothing,F]$ with $\rk_M(F)<t$.
	Thus $T_t(M)|F=M|F$ for every such flat, and for $0\le i<t$, we have
	\begin{equation}\label{eq:trun-mobi}
		w_i(T_t(M))=w_i(M) \quad \textrm{and} \quad \nu_i(T_t(M))=\nu_i(M).
	\end{equation}
	
	\subsection{Wonderful models and boundary divisors}
	\label{subsec:geometry-surface}
	
	All varieties in this section are defined over $\C$.
	We refer to~\cite{Hartshorne77} for the terminology and results from algebraic geometry used below.
	
	Fix an integer $i$ with $1\le i\le r-2$ and set $N=T_{i+2}(M)$.
	For $G\subseteq E$, let
	\[
	\operatorname{pr}_G:\C^E\longrightarrow\C^G,\qquad
	(z_e)_{e\in E}\longmapsto(z_e)_{e\in G}
	\]
	be the coordinate projection.
	A \emph{realization of $N$ over $\C$} is a linear subspace $V\subseteq\C^E$ such that
	\[
	\dim_{\C}\operatorname{pr}_G(V)=\rk_N(G)
	\qquad\text{for every }G\subseteq E.
	\]
	Assume that $N$ has a realization $V$ over $\C$.
	Then $\dim_{\C}V=\rk_N(E)=i+2$, so the projective space $\mathbb P(V)$ of one-dimensional subspaces of $V$ has dimension $i+1$.
	Similarly, $\mathbb P(\operatorname{pr}_G(V))$ has dimension $\rk_N(G)-1$ for every nonempty flat $G$.
	
	We recall the maximal projective wonderful model of De Concini and Procesi~\cite[Section~4]{DCP95}.
	For $0\ne z=(z_e)_{e\in E}\in V$, write $[z]=\C z=\{tz\mid t\in\C\}$.
	Let
	\[
	U=\{[z]\in\mathbb P(V)\mid z_e\ne 0\text{ for every }e\in E\}.
	\]
	Define the morphism
	\[
	\Phi:U\longrightarrow
	\prod_{\substack{G\in\mathcal L(N)\\G\ne\varnothing}}
	\mathbb P(\operatorname{pr}_G(V)),
	\qquad
	[z]\longmapsto\bigl([\operatorname{pr}_G(z)]\bigr)_{G\in\mathcal L(N)\setminus\{\varnothing\}}.
	\]
	The \emph{maximal projective wonderful model} $Y_N$ is the Zariski closure of $\Phi(U)$.
	
	Since $\operatorname{pr}_E(V)=V$, projection onto the $E$-coordinate defines the \emph{ordinary-coordinate map}
	\[
	p:Y_N\longrightarrow\mathbb P(V),\qquad (y_G)_{G\in\mathcal L(N)\setminus\{\varnothing\}}\longmapsto y_E.
	\]
	We use the following properties of $Y_N$; see~\cite[Section~4.2]{DCP95}.
	
	\begin{proposition}
		\label{prop:wonderful-model}
		
		\begin{enumerate}
			\item
			We have $p^{-1}(U)=\Phi(U)$, and the restriction
			\[
			p|_{\Phi(U)}:\Phi(U)\longrightarrow U
			\]
			is an isomorphism with inverse $\Phi$.
			
			\item
			$Y_N$ is a smooth irreducible projective variety of dimension $i+1$.
			
			\item
			The irreducible decomposition of the boundary is
			\[
			Y_N\setminus\Phi(U)
			=
			\bigcup_{G\in\mathcal L(N)\setminus\{\varnothing,E\}}D_G,
			\]
			where $D_G$ is the smooth irreducible divisor corresponding to $G$.
			
			\item
			For distinct nonempty proper flats $G_1,\ldots,G_k$ of $N$, we have $D_{G_1}\cap\cdots\cap D_{G_k}\ne\varnothing$ if and only if, after reordering, $G_1\subsetneq G_2\subsetneq\cdots\subsetneq G_k$.
			When this intersection is nonempty, the divisors meet transversely in a smooth subvariety of codimension $k$.
		\end{enumerate}
	\end{proposition}
	
	We identify $U$ with $\Phi(U)\subseteq Y_N$.
	The \emph{reciprocal-coordinate map} is the unique morphism
	\[
	q:Y_N\longrightarrow\mathbb P(\C^E)
	\]
	whose restriction to $U$ is given by
	\[
	q([z])=[(z_e^{-1})_{e\in E}]
	\quad\text{for every }[z]\in U,
	\]
	see~\cite[Remark~2.16]{BHMPW22} and~\cite[Example~4.1]{HK12}.
	The restriction $q|_U:U\longrightarrow q(U)$ is an isomorphism with inverse
	\[
	q(U)\longrightarrow U,\qquad
	[w]\longmapsto[(w_e^{-1})_{e\in E}].
	\]
	Since $U$ and $q(U)$ are dense open subsets of $Y_N$ and $q(Y_N)$, respectively,
	\[
	q:Y_N\longrightarrow q(Y_N)
	\]
	is birational.
	
	For a Cartier divisor $D$ on a variety $X$, write $[D]$ for its linear equivalence class.
	Let $W$ be a finite-dimensional complex vector space with $\dim_{\C}W\ge2$.
	A \emph{hyperplane} in $\mathbb P(W)$ is a projective linear subspace of codimension one.
	All hyperplanes in $\mathbb P(W)$ are linearly equivalent.
	Their common divisor class, denoted by $H_W$, is the \emph{hyperplane class} of $\mathbb P(W)$.
	
	Let $X$ be a smooth irreducible variety, and let $f:X\longrightarrow\mathbb P(W)$ be a morphism.
	Choose a hyperplane $H\subseteq\mathbb P(W)$ such that $f(X)\nsubseteq H$.
	The \emph{pullback} of $H_W$ is the divisor class
	\[
	f^*H_W=[f^{-1}(H)],
	\]
	where $f^{-1}(H)$ is taken with its multiplicities as an effective Cartier divisor.
	The class is independent of the choice of $H$.
	For the morphisms $p$ and $q$, define
	\[
	h=p^*H_V,\qquad k=q^*H_{\C^E}.
	\]
	
	Let $Y$ be a smooth projective variety of dimension $d$.
	For divisor classes $L_1,\ldots,L_d$ on $Y$, write $\degr_Y(L_1\cdots L_d)$ for the degree of their intersection product, or their \emph{intersection number}.
	
	Let
	\[
	\Gamma=
	\overline{
		\left\{
		\left([z],[(z_e^{-1})_{e\in E}]\right)
		\mid [z]\in U
		\right\}}
	\subseteq \mathbb P(V)\times\mathbb P(\C^E).
	\]
	The morphism $(p,q):Y_N\longrightarrow\Gamma$ is surjective and restricts to an isomorphism
	\[
	U\xrightarrow{\ \sim\ }
	\Gamma\cap\bigl(U\times\mathbb P(\C^E)\bigr).
	\]
	Thus $(p,q):Y_N\longrightarrow\Gamma$ is birational.
	
	\begin{proposition}[{\cite[Theorem~1.1]{HK12}}]
		\label{prop:geometry-mixed-degrees}
		Let $K$ be a loopless matroid of rank $s\ge2$ realizable over $\C$, and let $Y_K$ be the maximal projective wonderful model of a fixed realization of $K$.
		Let $h_K$ and $k_K$ denote the hyperplane pullback classes of the ordinary-coordinate and reciprocal-coordinate maps, respectively.
		For every $0\le j\le s-1$, we have
		\[
		\degr_{Y_K}\bigl(h_K^{s-1-j}k_K^j\bigr)=\nu_j(K).
		\]
	\end{proposition}
	
	Let $F$ be a nonempty proper flat of $N$.
	The matroids $N|F$ and $N/F$ are realized by $\operatorname{pr}_F(V)$ and $V\cap\C^{E\setminus F}$, respectively.
	Here we identify $\C^{E\setminus F}$ with the coordinate subspace
	\[
	\{(z_e)_{e\in E}\in\C^E\mid z_e=0\text{ for all }e\in F\}.
	\]
	Indeed, for $A\subseteq F$ and $B\subseteq E\setminus F$, we have
	\[
	\dim_{\C}\operatorname{pr}_A(\operatorname{pr}_F(V))
	=\rk_N(A)=\rk_{N|F}(A)
	\]
	and
	\[
	\dim_{\C}\operatorname{pr}_B(V\cap\C^{E\setminus F})
	=\rk_N(B\cup F)-\rk_N(F)
	=\rk_{N/F}(B).
	\]
	
	Let $Y_{N|F}$ and $Y_{N/F}$ denote the maximal projective wonderful models of $\operatorname{pr}_F(V)$ and $V\cap\C^{E\setminus F}$, respectively.
	
	\begin{proposition}[{\cite[Section~4.3]{DCP95}}]
		\label{prop:geometry-boundary-product}
		There is a natural isomorphism
		\begin{equation}\label{eq:geometry-boundary-product}
			D_F\simeq Y_{N|F}\times Y_{N/F}.
		\end{equation}
	\end{proposition}
	
	Under this identification, let
	\[
	\operatorname{pr}_1:D_F\longrightarrow Y_{N|F},
	\qquad
	\operatorname{pr}_2:D_F\longrightarrow Y_{N/F}
	\]
	denote the two projections.
	
	The two models have ordinary-coordinate and reciprocal-coordinate maps
	\[
	\begin{aligned}
		p_{N|F}&:Y_{N|F}\longrightarrow \mathbb P(\operatorname{pr}_F(V)),
		&
		q_{N|F}&:Y_{N|F}\longrightarrow\mathbb P(\C^F),\\
		p_{N/F}&:Y_{N/F}\longrightarrow
		\mathbb P(V\cap\C^{E\setminus F}),
		&
		q_{N/F}&:Y_{N/F}\longrightarrow\mathbb P(\C^{E\setminus F}).
	\end{aligned}
	\]
	Let $h_{N|F},k_{N|F},h_{N/F},k_{N/F}$ denote their hyperplane pullback classes, respectively.
	For a rank-one matroid, the wonderful model is a point; in this case, both divisor classes are taken to be zero.
	
	Let
	\[
	\iota_F:\mathbb P(V\cap\C^{E\setminus F})
	\hookrightarrow\mathbb P(V),
	\qquad
	j_F:\mathbb P(\C^F)\hookrightarrow\mathbb P(\C^E)
	\]
	be the natural inclusions, with $j_F$ obtained by setting the coordinates outside $F$ to zero.
	Under \eqref{eq:geometry-boundary-product}, the restrictions of $p$ and $q$ satisfy
	\begin{equation}\label{eq:geometry-boundary-maps}
		p|_{D_F}=\iota_F\circ p_{N/F}\circ\operatorname{pr}_2,
		\qquad
		q|_{D_F}=j_F\circ q_{N|F}\circ\operatorname{pr}_1.
	\end{equation}
	
	For the inclusion $\kappa_F:D_F\hookrightarrow Y_N$, the restriction of $[D_F]$ to $D_F$ is
	\[
	[D_F]|_{D_F}=\kappa_F^*[D_F].
	\]
	
	\begin{proposition}[{\cite[Proposition~2.24]{BHMPW22}}]
		\label{prop:geometry-boundary-restriction}
		Under the isomorphism \eqref{eq:geometry-boundary-product},
		we have
		\begin{align}
			h|_{D_F}
			&=\operatorname{pr}_2^*h_{N/F},
			\label{eq:geometry-boundary-h}\\
			k|_{D_F}
			&=\operatorname{pr}_1^*k_{N|F},
			\label{eq:geometry-boundary-k}\\
			[D_F]|_{D_F}
			&=-\operatorname{pr}_1^*h_{N|F}
			-\operatorname{pr}_2^*k_{N/F}.
			\label{eq:geometry-boundary-normal}
		\end{align}
		For every flat $H$ of $N$ with $F\subsetneq H\subsetneq E$, we also have
		\begin{equation}\label{eq:geometry-boundary-nested}
			[D_H]|_{D_F}
			=\operatorname{pr}_2^*[D_{H\setminus F}],
		\end{equation}
		where $D_{H\setminus F}$ is the boundary divisor of $Y_{N/F}$ corresponding to the flat $H\setminus F$.
	\end{proposition}
	
	Since $p_{N/F}$ is surjective and $q_{N|F}$ is birational onto its image, \eqref{eq:geometry-boundary-maps} gives
	\begin{equation}\label{eq:geometry-boundary-image-dimensions}
		\dim p(D_F)=i+1-\rk_N(F),
		\qquad
		\dim q(D_F)=\rk_N(F)-1.
	\end{equation}
	
	\subsection{The surface and boundary curves}
	\label{subsec:geometry-boundary-curves}
	
	Let $\Lambda\subseteq\mathbb P(\C^E)$ be a general linear subspace of codimension $i-1$, and set
	\begin{equation}\label{eq:geometry-surface}
		\widetilde S_i=q^{-1}(\Lambda).
	\end{equation}
	When $i=1$, we take $\Lambda=\mathbb P(\C^E)$, and hence $\widetilde S_i=Y_N$.
	
	A linear system $\mathcal D$ on a variety $X$ is \emph{base-point-free} if
	\[
	\bigcap_{D\in\mathcal D}\operatorname{Supp}(D)=\varnothing,
	\]
	where $\operatorname{Supp}(D)$ denotes the support of $D$.
	The divisors $q^{-1}(H)$, as $H$ ranges over the hyperplanes of $\mathbb P(\C^E)$ not containing $q(Y_N)$, form a linear system on $Y_N$.
	Given $y\in Y_N$, write $q(y)=[w]$, where $w=(w_e)_{e\in E}\ne0$, and choose $e\in E$ such that $w_e\ne0$.
	The hyperplane
	\[
	H_e=\{[x]\in\mathbb P(\C^E)\mid x_e=0\}
	\]
	satisfies $q(y)\notin H_e$.
	Hence $q(Y_N)\nsubseteq H_e$ and
	\[
	y\notin\operatorname{Supp}(q^{-1}(H_e)).
	\]
	Thus the linear system is base-point-free.
	We use the following form of Bertini's theorem~\cite{Jou83}.
	
	\begin{theorem}[{\cite{Jou83}}]\label{thm:geometry-bertini}
		Let $X$ be a smooth irreducible projective variety over $\C$, and let $f:X\longrightarrow\mathbb P^m$ be a morphism with $m\ge1$.
		For a general hyperplane $H\subseteq\mathbb P^m$, the following hold.
		\begin{enumerate}
			\item
			The inverse image $f^{-1}(H)$ is either empty or a smooth
			divisor on $X$.
			
			\item
			If $f$ is birational onto its image and $\dim X\ge2$, then
			$f^{-1}(H)$ is nonempty and irreducible, with
			\[
			\dim f^{-1}(H)=\dim X-1.
			\]
		\end{enumerate}
	\end{theorem}
	
	For $i\ge2$, write
	\[
	\Lambda=H_1\cap\cdots\cap H_{i-1},
	\]
	where $H_1,\ldots,H_{i-1}$ are general hyperplanes.
	Set $Y_0=Y_N$ and define
	\[
	Y_j=Y_{j-1}\cap q^{-1}(H_j)
	\quad\text{for all }1\le j\le i-1.
	\]
	
	We show by induction that each $Y_j$ is smooth and irreducible, with
	\[
	Y_j\cap U\ne\varnothing,
	\qquad
	\dim Y_j=i+1-j
	\quad\text{for all }0\le j\le i-1.
	\]
	The case $j=0$ follows from Proposition~\ref{prop:wonderful-model}.
	Suppose that the assertions hold for $Y_{j-1}$, where $1\le j\le i-1$.
	Since $Y_{j-1}\cap U$ is dense in $Y_{j-1}$ and $q|_U$ is an isomorphism onto $q(U)$, the restriction
	\[
	q|_{Y_{j-1}}:Y_{j-1}\longrightarrow q(Y_{j-1})
	\]
	is birational.
	Moreover,
	\[
	\dim Y_{j-1}=i+2-j\ge3.
	\]
	The variety $Y_{j-1}$ is projective because it is closed in $Y_N$.
	Theorem~\ref{thm:geometry-bertini} therefore shows that, for a general $H_j$, the variety $Y_j$ is smooth and irreducible, with
	\[
	\dim Y_j=\dim Y_{j-1}-1=i+1-j.
	\]
	We may also choose $H_j$ so that
	\[
	H_j\cap q(Y_{j-1}\cap U)\ne\varnothing,
	\]
	since $q(Y_{j-1}\cap U)$ is dense and open in $q(Y_{j-1})$.
	Hence
	\[
	Y_j\cap U
	=(Y_{j-1}\cap U)\cap q^{-1}(H_j)
	\ne\varnothing.
	\]
	This completes the induction.
	
	It follows that $\widetilde S_i=Y_{i-1}$ is a smooth irreducible projective surface meeting $U$.
	For $i=1$, this follows directly from Proposition~\ref{prop:wonderful-model}.
	
	Let $F_1,\ldots,F_s$ be distinct nonempty proper flats of $N$ such that
	\[
	Z=D_{F_1}\cap\cdots\cap D_{F_s}\ne\varnothing.
	\]
	Proposition~\ref{prop:wonderful-model} shows that $Z$ is smooth of codimension $s$ in $Y_N$.
	Applying Theorem~\ref{thm:geometry-bertini}~\textup{(i)} successively to the restrictions of $q$ to the irreducible components of $Z$, we find that $\widetilde S_i\cap Z$ is either empty or smooth for a general $\Lambda$.
	If it is nonempty, then
	\[
	\dim(\widetilde S_i\cap Z)
	=\dim Z-(i-1)
	=2-s.
	\]
	Thus $\widetilde S_i$ and $Z$ intersect transversely whenever their intersection is nonempty.
	There are only finitely many intersections of boundary divisors.
	We may therefore choose $\Lambda$ so that $\widetilde S_i$ meets each $D_{F_1}\cap\cdots\cap D_{F_s}$ transversely whenever
	\[
	\widetilde S_i\cap D_{F_1}\cap\cdots\cap D_{F_s}\ne\varnothing.
	\]
	
	The restrictions of $h$ and $k$ to $\widetilde S_i$ are
	\[
	h|_{\widetilde S_i}
	=(p|_{\widetilde S_i})^*H_V,
	\qquad
	k|_{\widetilde S_i}
	=(q|_{\widetilde S_i})^*H_{\C^E}.
	\]
	For divisor classes $L,L'$ on $\widetilde S_i$, write
	\[
	L\cdot L'=\degr_{\widetilde S_i}(LL'),
	\qquad
	L^2=L\cdot L.
	\]
	For a divisor class $L$ on $\widetilde S_i$ and a smooth irreducible curve $C\subseteq\widetilde S_i$, we have
	\[
	L\cdot[C]=\degr_C(L|_C).
	\]
	Since $\widetilde S_i$ is the intersection of $i-1$ hyperplane pullbacks of class $k$, Proposition~\ref{prop:geometry-mixed-degrees} gives
	\begin{equation}\label{eq:geometry-surface-degrees}
		\begin{aligned}
			\bigl(h|_{\widetilde S_i}\bigr)^2
			&=\degr_{Y_N}(h^2k^{i-1})
			=\nu_{i-1}(M)>0,\\
			h|_{\widetilde S_i}\cdot k|_{\widetilde S_i}
			&=\degr_{Y_N}(hk^i)
			=\nu_i(M),\\
			\bigl(k|_{\widetilde S_i}\bigr)^2
			&=\degr_{Y_N}(k^{i+1})
			=\nu_{i+1}(M).
		\end{aligned}
	\end{equation}
	
	Let $F$ be a nonempty proper flat of $N$.
	The surface $\widetilde S_i$ meets $U$, so it is not contained in $D_F$.
	The restriction $D_F|_{\widetilde S_i}$ is therefore an effective Cartier divisor with
	\[
	\operatorname{Supp}(D_F|_{\widetilde S_i})
	=D_F\cap\widetilde S_i.
	\]
	If $D_F\cap\widetilde S_i=\varnothing$, then $D_F|_{\widetilde S_i}=0$.
	Otherwise, transversality implies that every irreducible component of $D_F|_{\widetilde S_i}$ is a smooth curve of multiplicity one.
	We call the irreducible components of $D_F|_{\widetilde S_i}$ the \emph{branches over $F$}.
	
	Truncation preserves the flats of ranks $i$ and $i+1$ and the restrictions to these flats.
	We therefore index the boundary curves by the corresponding flats of $M$.
	
	\begin{proposition}\label{prop:geometry-boundary-curves}
		\begin{enumerate}
			\item
			If $F$ is a nonempty proper flat of $N$ with $\rk_N(F)<i$,
			then
			\[
			D_F\cap\widetilde S_i=\varnothing.
			\]
			
			\item
			For every $F\in\F_i(M)$,
			\begin{equation}\label{eq:geometry-lower-branches}
				D_F|_{\widetilde S_i}
				=
				\sum_{t=1}^{\nu_{i-1}(M|F)}C_{F,t},
			\end{equation}
			where the curves $C_{F,t}$ are pairwise disjoint and each is isomorphic to $\mathbb P^1$.
			
			\item
			For every $H\in\F_{i+1}(M)$,
			\[
			D_H|_{\widetilde S_i}=E_H,
			\]
			where the curves $E_H$ are smooth, irreducible, and pairwise disjoint.
		\end{enumerate}
	\end{proposition}
	
	\begin{proof}
		Let $F$ be a nonempty proper flat of $N$ with $\rk_N(F)<i$.
		By \eqref{eq:geometry-boundary-image-dimensions},
		\[
		\dim q(D_F)=\rk_N(F)-1<i-1.
		\]
		Since $\Lambda$ has codimension $i-1$ and is general,
		\[
		\Lambda\cap q(D_F)=\varnothing.
		\]
		Therefore,
		\[
		D_F\cap\widetilde S_i
		=D_F\cap q^{-1}(\Lambda)
		=\varnothing.
		\]
		This proves \textup{(i)}.
		
		Now let $F\in\F_i(M)$.
		Since $N|F=M|F$ and $\rk(N/F)=2$, we have
		\[
		\dim Y_{N|F}=i-1,
		\qquad
		\dim Y_{N/F}=1.
		\]
		Equations~\eqref{eq:geometry-boundary-product} and~\eqref{eq:geometry-boundary-maps} give
		\[
		D_F\cap\widetilde S_i
		=(q|_{D_F})^{-1}(\Lambda)
		\simeq
		q_{N|F}^{-1}\bigl(j_F^{-1}(\Lambda)\bigr)
		\times Y_{N/F}.
		\]
		
		Suppose that $i\ge2$.
		Write $\Lambda=H_1\cap\cdots\cap H_{i-1}$.
		Then
		\[
		j_F^{-1}(\Lambda)
		=
		j_F^{-1}(H_1)\cap\cdots\cap j_F^{-1}(H_{i-1}),
		\]
		where $j_F^{-1}(H_1),\ldots,j_F^{-1}(H_{i-1})$ are general hyperplanes in $\mathbb P(\C^F)$.
		Applying Theorem~\ref{thm:geometry-bertini}~\textup{(i)} successively shows that
		\[
		q_{N|F}^{-1}\bigl(j_F^{-1}(\Lambda)\bigr)
		\]
		is either empty or consists of finitely many reduced points.
		
		Since the points are reduced, their number is
		\[
		\left|q_{N|F}^{-1}\bigl(j_F^{-1}(\Lambda)\bigr)\right|
		=\degr_{Y_{N|F}}\bigl(k_{N|F}^{i-1}\bigr)=\nu_{i-1}(N|F)
		=\nu_{i-1}(M|F)>0,
		\]
		where the second equality follows from Proposition~\ref{prop:geometry-mixed-degrees} with $K=N|F$.
		For $i=1$, we have $\Lambda=\mathbb P(\C^E)$ and $q_{N|F}^{-1}\bigl(j_F^{-1}(\Lambda)\bigr)=Y_{N|F}$.
		In this case, $Y_{N|F}$ is a point and $\nu_0(M|F)=1$.
		Write
		\[
		q_{N|F}^{-1}\bigl(j_F^{-1}(\Lambda)\bigr)
		=
		\{u_{F,1},\ldots,u_{F,\nu_{i-1}(M|F)}\},
		\]
		where the $u_{F,t}$ are distinct points of $Y_{N|F}$.
		
		Since $N/F$ has rank two, every nonempty proper flat of $N/F$ has rank one.
		For each such flat $G$, the projective space
		\[
		\mathbb P\bigl(\operatorname{pr}_G(V\cap\C^{E\setminus F})\bigr)
		\]
		is a point, since its dimension is $\rk_{N/F}(G)-1=0$.
		Hence
		\[
		Y_{N/F}
		\simeq
		\mathbb P(V\cap\C^{E\setminus F})
		\simeq\mathbb P^1.
		\]
		For $1\le t\le\nu_{i-1}(M|F)$, set
		\[
		C_{F,t}=\{u_{F,t}\}\times Y_{N/F}.
		\]
		The curves are pairwise disjoint, and each is isomorphic to $\mathbb P^1$.
		They occur with multiplicity one because $\widetilde S_i$ and $D_F$ meet transversely.
		Thus
		\[
		D_F|_{\widetilde S_i}
		=
		\sum_{t=1}^{\nu_{i-1}(M|F)}C_{F,t}.
		\]
		
		Distinct flats $F,F'\in\F_i(M)$ are incomparable, so Proposition~\ref{prop:wonderful-model} gives $D_F\cap D_{F'}=\varnothing$.
		Thus $C_{F,t}$ and $C_{F',t'}$ are also disjoint whenever $F\ne F'$.
		This proves \textup{(ii)}.
		
		Finally, let $H\in\F_{i+1}(M)$.
		Since $N/H$ has rank one, its wonderful model is a point, and \eqref{eq:geometry-boundary-product} gives
		\[
		D_H\simeq Y_{N|H}\times Y_{N/H}\simeq Y_{N|H}.
		\]
		Under this identification,
		\[
		q|_{D_H}=j_H\circ q_{N|H}
		\]
		by \eqref{eq:geometry-boundary-maps}.
		Since $\widetilde S_i=q^{-1}(\Lambda)$, we have
		\[
		D_H\cap\widetilde S_i
		=(q|_{D_H})^{-1}(\Lambda)
		=q_{N|H}^{-1}\bigl(j_H^{-1}(\Lambda)\bigr).
		\]
		
		The variety $Y_{N|H}$ is smooth, irreducible, and projective, with $\dim Y_{N|H}=\rk(N|H)-1=i$.
		The morphism $q_{N|H}$ is birational onto its image.
		For $i\ge2$, write $\Lambda=H_1\cap\cdots\cap H_{i-1}$.
		Then
		\[
		j_H^{-1}(\Lambda)
		=
		j_H^{-1}(H_1)\cap\cdots\cap j_H^{-1}(H_{i-1}),
		\]
		where $j_H^{-1}(H_1),\ldots,j_H^{-1}(H_{i-1})$ are general hyperplanes in $\mathbb P(\C^H)$.
		As in the construction of $\widetilde S_i$, successive applications of Theorem~\ref{thm:geometry-bertini} show that
		\[
		q_{N|H}^{-1}\bigl(j_H^{-1}(\Lambda)\bigr)
		\]
		is smooth and irreducible, with
		\[
		\dim q_{N|H}^{-1}\bigl(j_H^{-1}(\Lambda)\bigr)
		=i-(i-1)=1.
		\]
		For $i=1$, we have $\Lambda=\mathbb P(\C^E)$, and the inverse image is $Y_{N|H}$, which is a smooth irreducible curve.
		
		Set
		\[
		E_H=D_H\cap\widetilde S_i.
		\]
		Since $\widetilde S_i$ and $D_H$ intersect transversely, we have
		\[
		D_H|_{\widetilde S_i}=E_H.
		\]
		Distinct flats $H,H'\in\F_{i+1}(M)$ are incomparable, so Proposition~\ref{prop:wonderful-model} gives
		\[
		D_H\cap D_{H'}=\varnothing.
		\]
		Hence $E_H\cap E_{H'}=\varnothing$.
		This proves \textup{(iii)}.
	\end{proof}
	
	\begin{proposition}\label{prop:geometry-boundary-intersections}
		For $F\in\F_i(M)$, $H\in\F_{i+1}(M)$, and $1\le t\le\nu_{i-1}(M|F)$, the intersection numbers are as follows:
		\begin{equation}\label{eq:geometry-boundary-intersections}
			\begin{array}{c|cc}
				& C_{F,t} & E_H\\ \hline
				h|_{\widetilde S_i}
				&1
				&0\\[2pt]
				k|_{\widetilde S_i}
				&0
				&\nu_i(M|H)\\[2pt]
				C_{F,t}
				&-\bigl(|\F_1(M/F)|-1\bigr)
				&\mathbf 1_{\{F\subsetneq H\}}\\[2pt]
				E_H
				&\mathbf 1_{\{F\subsetneq H\}}
				&-\nu_{i-1}(M|H)
			\end{array}
		\end{equation}
		Each entry is the intersection number of the classes indicated by its row and column.
		The indicator is defined by
		\[
		\mathbf 1_{\{F\subsetneq H\}}
		=
		\begin{cases}
			1,&F\subsetneq H,\\
			0,&F\not\subseteq H.
		\end{cases}
		\]
		If $F\subsetneq H$, then $C_{F,t}$ and $E_H$ intersect transversely at exactly one point.
	\end{proposition}
	
	\begin{proof}
		Let $F\in\F_i(M)$ and $1\le t\le\nu_{i-1}(M|F)$.
		Under \eqref{eq:geometry-boundary-product}, we have
		\[
		C_{F,t}=\{u_{F,t}\}\times Y_{N/F}.
		\]
		Thus
		\[
		\operatorname{pr}_1(C_{F,t})=\{u_{F,t}\},
		\]
		and the restriction
		\[
		\operatorname{pr}_2|_{C_{F,t}}:
		C_{F,t}\longrightarrow Y_{N/F}\simeq\mathbb P^1
		\]
		is an isomorphism.
		
		Under the identification $C_{F,t}\simeq Y_{N/F}$, Proposition~\ref{prop:geometry-boundary-restriction} gives
		\[
		h|_{C_{F,t}}=h_{N/F},
		\qquad
		k|_{C_{F,t}}=0,
		\qquad
		[D_F]|_{C_{F,t}}=-k_{N/F}.
		\]
		By Proposition~\ref{prop:geometry-mixed-degrees}, we have
		\[
		h|_{\widetilde S_i}\cdot[C_{F,t}]
		=\degr_{Y_{N/F}}(h_{N/F})=1,
		\qquad k|_{\widetilde S_i}\cdot[C_{F,t}]=0.
		\]
		Since the curves $C_{F,t}$ are pairwise disjoint, we have
		\[
		[C_{F,t}]^2
		=[D_F]|_{\widetilde S_i}\cdot[C_{F,t}]=\degr_{C_{F,t}}\bigl([D_F]|_{C_{F,t}}\bigr)=-\degr_{Y_{N/F}}(k_{N/F})
		=-\nu_1(N/F).
		\]
		Since $\F_1(N/F)=\F_1(M/F)$, \eqref{eq:w-nu} gives
		\[
		\nu_1(N/F)=w_1(N/F)-1=|\F_1(M/F)|-1.
		\]
		Therefore,
		\[
		[C_{F,t}]^2=-\bigl(|\F_1(M/F)|-1\bigr).
		\]
		
		Let $H\in\F_{i+1}(M)$.
		Since $Y_{N/H}$ is a point, we identify $D_H$ with
		$Y_{N|H}$.
		By Proposition~\ref{prop:geometry-boundary-restriction},
		\[
		h|_{D_H}=0,
		\qquad
		k|_{D_H}=k_{N|H},
		\qquad
		[D_H]|_{D_H}=-h_{N|H}.
		\]
		The curve $E_H$ is the intersection of $i-1$ divisors of class $k_{N|H}$ on $Y_{N|H}$, and
		$D_H|_{\widetilde S_i}=E_H$.
		By Proposition~\ref{prop:geometry-mixed-degrees}
		and $N|H=M|H$, we have
		\begin{align*}
			h|_{\widetilde S_i}\cdot[E_H]
			&=0,\\
			k|_{\widetilde S_i}\cdot[E_H]
			&=\degr_{Y_{N|H}}\bigl(k_{N|H}^{i}\bigr)
			=\nu_i(M|H),\\
			[E_H]^2
			&=\degr_{E_H}\bigl([D_H]|_{E_H}\bigr)=-\degr_{Y_{N|H}}
			\bigl(h_{N|H}k_{N|H}^{i-1}\bigr)
			=-\nu_{i-1}(M|H).
		\end{align*}
		
		If $F\not\subseteq H$, then $F$ and $H$ are incomparable.
		By Proposition~\ref{prop:wonderful-model},
		$D_F\cap D_H=\varnothing$, and hence
		\[
		[C_{F,t}]\cdot[E_H]=0.
		\]
		Suppose that $F\subsetneq H$.
		By \eqref{eq:geometry-boundary-nested},
		\[
		[D_H]|_{D_F}
		=\operatorname{pr}_2^*[D_{H\setminus F}].
		\]
		Since $N/F$ has rank two, the boundary divisor
		$D_{H\setminus F}$ is a point of
		$Y_{N/F}\simeq\mathbb P^1$.
		The restriction of $\operatorname{pr}_2$ to $C_{F,t}$
		is an isomorphism, so
		\[
		[C_{F,t}]\cdot[E_H]
		=\degr_{C_{F,t}}\bigl([D_H]|_{C_{F,t}}\bigr)=\degr_{Y_{N/F}}[D_{H\setminus F}]
		=1.
		\]
		The curves meet transversely by Proposition~\ref{prop:wonderful-model}.
		Their intersection number is one, so they meet at exactly one point.
	\end{proof}
	
	\begin{remark}\label{rem:geometry-whitney-numbers}
		For each $F\in\F_i(M)$, the divisor $D_F|_{\widetilde S_i}$ has exactly $\nu_{i-1}(M|F)$ irreducible components.
		By the definition of the characteristic polynomial, the identity $\mathcal L(M|F)=[\varnothing,F]$, and \eqref{eq:w-nu}, we have
		\[
		w_i(M)
		=\sum_{F\in\F_i(M)}w_i(M|F)
		=\sum_{F\in\F_i(M)}\nu_{i-1}(M|F).
		\]
		Thus $W_i(M)$ counts the flats indexing these divisors, whereas $w_i(M)$ counts their irreducible components $C_{F,t}$.
	\end{remark}
	
	\subsection{Numerical pullback and the weighted inequality}
	\label{subsec:geometry-pullback}
	
	Since $p|_U$ is an isomorphism and $\widetilde S_i\cap U\ne\varnothing$, the restriction
	\[
	p|_{\widetilde S_i}:
	\widetilde S_i\longrightarrow p(\widetilde S_i)
	\]
	is birational.
	Thus $p(\widetilde S_i)$ is an irreducible projective surface.
	
	A variety is \emph{normal} if all its local rings are integrally closed domains.
	Let $X$ be an irreducible variety.
	For an affine open subset of $X$ with coordinate ring $A$, let $\overline A$ be the integral closure of $A$ in its field of fractions.
	The affine varieties defined by these rings glue to a normal irreducible variety $X^\nu$, called the \emph{normalization} of $X$.
	The inclusions $A\subseteq\overline A$ induce the \emph{normalization morphism} $X^\nu\longrightarrow X$.
	
	Let $S_i$ be the normalization of $p(\widetilde S_i)$, with the natural finite birational morphism
	\[
	\rho_i:S_i\longrightarrow p(\widetilde S_i).
	\]
	The surface $\widetilde S_i$ is smooth and hence normal.
	The universal property of normalization gives a unique morphism
	\[
	\pi_i:\widetilde S_i\longrightarrow S_i
	\]
	such that
	\begin{equation}\label{eq:geometry-normalization-factorization}
		\rho_i\circ\pi_i=p|_{\widetilde S_i}.
	\end{equation}
	The morphism $\pi_i$ is proper and birational.
	These properties of normalization can be found in~\cite{Hartshorne77}.
	
	An irreducible curve $C\subseteq\widetilde S_i$ is \emph{exceptional} for $\pi_i$ if $\pi_i(C)$ is a point.
	
	\begin{proposition}\label{prop:geometry-exceptional-curves}
		The exceptional curves of $\pi_i$ are precisely the curves $E_H$ with $H\in\F_{i+1}(M)$.
	\end{proposition}
	
	\begin{proof}
		Let $C\subseteq\widetilde S_i$ be an irreducible curve.
		Since $\rho_i$ is finite and $\rho_i\circ\pi_i=p|_{\widetilde S_i}$, the image $\pi_i(C)$ is a point if and only if $p(C)$ is a point.
		Moreover,
		\[
		h|_{\widetilde S_i}
		=(p|_{\widetilde S_i})^*H_V,
		\]
		so
		\[
		h|_{\widetilde S_i}\cdot[C] =\degr_C\bigl((p|_C)^*H_V\bigr).
		\]
		This degree is zero when $p(C)$ is a point and strictly positive when $p(C)$ is a curve.
		Hence $C$ is exceptional if and only if $h|_{\widetilde S_i}\cdot[C]=0$.
		
		Suppose that $C$ is exceptional, so that $p(C)$ is a point.
		Since $p|_U$ is an isomorphism, we have $C\cap U=\varnothing$.
		By irreducibility and Proposition~\ref{prop:geometry-boundary-curves}, $C$ is one of the curves $C_{F,t}$ or $E_H$.
		By Proposition~\ref{prop:geometry-boundary-intersections},
		\[
		h|_{\widetilde S_i}\cdot[C_{F,t}]=1,
		\qquad
		h|_{\widetilde S_i}\cdot[E_H]=0.
		\]
		Thus the exceptional curves are precisely the $E_H$ with $H\in\F_{i+1}(M)$.
	\end{proof}
	
	Since $S_i$ is normal and $\pi_i$ is proper and birational, there is an open subset $S_i^\circ\subseteq S_i$, with finite complement, over which
	\[
	\pi_i|_{\pi_i^{-1}(S_i^\circ)}:
	\pi_i^{-1}(S_i^\circ)\longrightarrow S_i^\circ
	\]
	is an isomorphism.
	For a prime divisor $B$ on $S_i$, its \emph{strict transform} on $\widetilde S_i$ is
	\[
	\widetilde B
	=
	\overline{\pi_i^{-1}(B\cap S_i^\circ)},
	\]
	where the Zariski closure is taken in $\widetilde S_i$.
	For a real Weil divisor $D=\sum_j a_jB_j$ on $S_i$, its strict transform is
	\[
	\widetilde D=\sum_j a_j\widetilde B_j.
	\]
	
	For an irreducible curve $C\subseteq\widetilde S_i$, define its \emph{pushforward} under $\pi_i$ by
	\[
	(\pi_i)_*C=
	\begin{cases}
		\pi_i(C),&\dim\pi_i(C)=1,\\
		0,&\dim\pi_i(C)=0.
	\end{cases}
	\]
	We extend this definition linearly to real Weil divisors.
	
	Every nonexceptional curve is the strict transform of its image, and distinct nonexceptional curves have distinct images, because $\pi_i$ is an isomorphism over $S_i^\circ$.
	Thus Proposition~\ref{prop:geometry-exceptional-curves} shows that the curves $\pi_i(C_{F,t})$ are distinct prime divisors on $S_i$, with strict transforms $C_{F,t}$.
	
	We use Mumford's numerical pullback~\cite[pp.~17--18]{Mumford61}, extended linearly to real Weil divisors.
	
	\begin{definition}\label{def:geometry-numerical-pullback}
		Let $D$ be a real Weil divisor on $S_i$, and let
		$\widetilde D$ be its strict transform on $\widetilde S_i$.
		The \emph{numerical pullback} of $D$ under $\pi_i$ is the real divisor
		\[
		\pi_i^{\mathrm{num}*}D
		=
		\widetilde D+
		\sum_{H\in\F_{i+1}(M)}c_HE_H
		\]
		satisfying
		\[
		[\pi_i^{\mathrm{num}*}D]\cdot[E_H]=0
		\quad\text{for every }H\in\F_{i+1}(M).
		\]
	\end{definition}
	
	By Proposition~\ref{prop:geometry-boundary-intersections}, the curves $E_H$ are pairwise disjoint and
	\[
	[E_H]^2=-\nu_{i-1}(M|H)<0.
	\]
	Hence the coefficients $c_H$ exist and are uniquely determined by
	\[
	c_H=
	\frac{[\widetilde D]\cdot[E_H]}{\nu_{i-1}(M|H)}
	\quad\text{for all }H\in\F_{i+1}(M).
	\]
	
	Let $a=(a_F)_{F\in\F_i(M)}$ be a family of real weights.
	Assign the coefficient $a_F/\nu_{i-1}(M|F)$ to each curve over $F$, and define
	\begin{equation}\label{eq:geometry-weighted-divisors}
		\widetilde D(a)
		=
		\sum_{F\in\F_i(M)}
		\frac{a_F}{\nu_{i-1}(M|F)}
		\sum_{t=1}^{\nu_{i-1}(M|F)}C_{F,t},
		\qquad
		D(a)=(\pi_i)_*\widetilde D(a).
	\end{equation}
	Since the curves $C_{F,t}$ are nonexceptional and have distinct images, $\widetilde D(a)$ is the strict transform of $D(a)$.
	
	\begin{proposition}\label{prop:geometry-numerical-pullback}
		The numerical pullback of $D(a)$ is
		\begin{equation}\label{eq:geometry-numerical-pullback}
			D^\#(a):=\pi_i^{\mathrm{num}*}D(a)
			=
			\widetilde D(a)
			+
			\sum_{H\in\F_{i+1}(M)}
			\frac{\displaystyle\sum_{F\in\F_i(M|H)}a_F}
			{\nu_{i-1}(M|H)}E_H.
		\end{equation}
		Moreover,
		\begin{equation}\label{eq:geometry-pullback-degree}
			h|_{\widetilde S_i}\cdot[D^\#(a)]
			=
			\sum_{F\in\F_i(M)}a_F,
		\end{equation}
		and
		\begin{equation}\label{eq:geometry-pullback-square}
			[D^\#(a)]^2
			=
			-\sum_{F\in\F_i(M)}
			\frac{|\F_1(M/F)|-1}{\nu_{i-1}(M|F)}a_F^2
			+
			\sum_{H\in\F_{i+1}(M)}
			\frac{\left(\displaystyle\sum_{F\in\F_i(M|H)}a_F\right)^2}
			{\nu_{i-1}(M|H)}.
		\end{equation}
	\end{proposition}
	
	\begin{proof}
		By Definition~\ref{def:geometry-numerical-pullback}, we may write
		\[
		D^\#(a)
		=
		\widetilde D(a)
		+
		\sum_{H\in\F_{i+1}(M)}c_HE_H.
		\]
		For $F\in\F_i(M)$, $H\in\F_{i+1}(M)$, and $1\le t\le\nu_{i-1}(M|F)$, Proposition~\ref{prop:geometry-boundary-intersections} gives
		\[
		[C_{F,t}]\cdot[E_H]
		=
		\begin{cases}
			1,&F\subsetneq H,\\
			0,&F\not\subseteq H.
		\end{cases}
		\]
		Therefore,
		\[
		[\widetilde D(a)]\cdot[E_H]
		=
		\sum_{F\in\F_i(M|H)}
		\frac{a_F}{\nu_{i-1}(M|F)}
		\sum_{t=1}^{\nu_{i-1}(M|F)}
		[C_{F,t}]\cdot[E_H]=\sum_{F\in\F_i(M|H)}a_F.
		\]
		Since the curves $E_H$ are pairwise disjoint and $[E_H]^2=-\nu_{i-1}(M|H)$, we have
		\[
		0=[D^\#(a)]\cdot[E_H]
		=
		\sum_{F\in\F_i(M|H)}a_F
		-c_H\nu_{i-1}(M|H).
		\]
		Thus
		\[
		c_H=
		\frac{\displaystyle\sum_{F\in\F_i(M|H)}a_F}
		{\nu_{i-1}(M|H)}
		\quad\text{for all }H\in\F_{i+1}(M),
		\]
		which proves \eqref{eq:geometry-numerical-pullback}.
		
		By Proposition~\ref{prop:geometry-boundary-intersections}, we have
		\[
		h|_{\widetilde S_i}\cdot[C_{F,t}]=1,
		\qquad
		h|_{\widetilde S_i}\cdot[E_H]=0.
		\]
		Therefore,
		\[
		h|_{\widetilde S_i}\cdot[D^\#(a)]
		=h|_{\widetilde S_i}\cdot[\widetilde D(a)]=
		\sum_{F\in\F_i(M)}
		\frac{a_F}{\nu_{i-1}(M|F)}
		\sum_{t=1}^{\nu_{i-1}(M|F)}
		h|_{\widetilde S_i}\cdot[C_{F,t}]=\sum_{F\in\F_i(M)}a_F.
		\]
		This proves \eqref{eq:geometry-pullback-degree}.
		
		The curves $C_{F,t}$ are pairwise disjoint.
		By Proposition~\ref{prop:geometry-boundary-intersections}, we have
		\begin{equation}\label{eq:geometry-strict-transform-square}
			\begin{aligned}
				[\widetilde D(a)]^2
				&=
				\sum_{F\in\F_i(M)}
				\left(\frac{a_F}{\nu_{i-1}(M|F)}\right)^2
				\sum_{t=1}^{\nu_{i-1}(M|F)}[C_{F,t}]^2\\
				&=
				-\sum_{F\in\F_i(M)}
				\frac{|\F_1(M/F)|-1}{\nu_{i-1}(M|F)}a_F^2.
			\end{aligned}
		\end{equation}
		By the definition of numerical pullback,
		\[
		[D^\#(a)]\cdot[E_H]=0
		\quad\text{for all }H\in\F_{i+1}(M).
		\]
		Thus, by \eqref{eq:geometry-numerical-pullback},
		\begin{equation}\label{eq:geometry-pullback-square-expansion}
			\begin{aligned}
				[D^\#(a)]^2
				&=[D^\#(a)]\cdot[\widetilde D(a)]\\
				&=[\widetilde D(a)]^2
				+\sum_{H\in\F_{i+1}(M)}
				\frac{\displaystyle\sum_{F\in\F_i(M|H)}a_F}
				{\nu_{i-1}(M|H)}
				[\widetilde D(a)]\cdot[E_H]\\
				&=[\widetilde D(a)]^2
				+\sum_{H\in\F_{i+1}(M)}
				\frac{\left(\displaystyle\sum_{F\in\F_i(M|H)}a_F\right)^2}
				{\nu_{i-1}(M|H)}.
			\end{aligned}
		\end{equation}
		Substituting \eqref{eq:geometry-strict-transform-square} into \eqref{eq:geometry-pullback-square-expansion} proves \eqref{eq:geometry-pullback-square}.
	\end{proof}
	
	Let $X$ be a smooth irreducible projective surface.
	Two real divisor classes $L$ and $L'$ on $X$ are \emph{numerically equivalent} if
	\[
	(L-L')\cdot[C]=0
	\quad\text{for every irreducible curve }C\subseteq X.
	\]
	Let $N^1(X)_{\R}$ denote the real vector space of divisor classes modulo numerical equivalence.
	The intersection product
	\[
	N^1(X)_{\R}\times N^1(X)_{\R}\longrightarrow\R,
	\qquad
	(L,L')\longmapsto L\cdot L'
	\]
	is symmetric and bilinear.
	
	\begin{theorem}[Hodge index theorem
		{\cite[Chapter~V, Theorem~1.9 and Remark~1.9.1]{Hartshorne77}}]
		\label{thm:geometry-hodge-index}
		Let $X$ be a smooth irreducible projective surface, and set $\rho=\dim_{\R}N^1(X)_{\R}$.
		There exists a basis $L_1,\ldots,L_\rho$ of $N^1(X)_{\R}$ such that
		\[
		A=(L_j\cdot L_k)_{1\le j,k\le\rho}
		=\operatorname{diag}(1,-1,\ldots,-1).
		\]
	\end{theorem}
	
	\begin{theorem}\label{thm:geometry-weighted-inequality}
		Let $M$ be a loopless matroid of rank $r$, and let $1\le i\le r-2$.
		Assume that the truncation $T_{i+2}(M)$ is realizable over $\C$.
		Then \eqref{ineq:weight} holds for every family of real weights $(a_F)_{F\in\F_i(M)}$.
	\end{theorem}
	
	\begin{proof}
		Let $a=(a_F)_{F\in\F_i(M)}$ be a family of real weights.
		By \eqref{eq:geometry-surface-degrees} and
		\eqref{eq:geometry-pullback-degree}, we have
		\[
		\bigl(h|_{\widetilde S_i}\bigr)^2
		=\nu_{i-1}(M)>0,
		\qquad
		h|_{\widetilde S_i}\cdot[D^\#(a)]
		=\sum_{F\in\F_i(M)}a_F.
		\]
		Set
		\[
		L=
		[D^\#(a)]
		-
		\frac{\displaystyle\sum_{F\in\F_i(M)}a_F}
		{\nu_{i-1}(M)}
		h|_{\widetilde S_i}
		\in N^1(\widetilde S_i)_{\R}.
		\]
		Then
		\[
		h|_{\widetilde S_i}\cdot L=0.
		\]
		If $L^2>0$, then the intersection form is positive definite on the two-dimensional subspace spanned by $h|_{\widetilde S_i}$ and $L$.
		This contradicts Theorem~\ref{thm:geometry-hodge-index}, so $L^2\le0$.
		
		By the definition of $L$, we have
		\begin{equation}\label{eq:geometry-hodge-bound}
			0\ge L^2
			=
			[D^\#(a)]^2
			-
			\frac{\left(\displaystyle\sum_{F\in\F_i(M)}a_F\right)^2}
			{\nu_{i-1}(M)}.
		\end{equation}
		Substituting \eqref{eq:geometry-pullback-square} into \eqref{eq:geometry-hodge-bound} gives
		\[
		\sum_{H\in\F_{i+1}(M)}
		\frac{\left(\displaystyle\sum_{F\in\F_i(M|H)}a_F\right)^2}
		{\nu_{i-1}(M|H)}
		\le
		\frac{\left(\displaystyle\sum_{F\in\F_i(M)}a_F\right)^2}
		{\nu_{i-1}(M)}
		+
		\sum_{F\in\F_i(M)}
		\frac{|\F_1(M/F)|-1}{\nu_{i-1}(M|F)}a_F^2.
		\]
		This completes the proof.
	\end{proof}
	
	\begin{remark}
		For $i=r-1$, we have $\F_r(M)=\{E\}$ and $|\F_1(M/F)|=1$ for every $F\in\F_{r-1}(M)$.
		Thus \eqref{ineq:weight} holds with equality, without any realizability assumption.
	\end{remark}
	
	\section{The descended algebra and the weighted inequality}
	\label{sec:weighted}
	
	Let $M$ be a loopless matroid of rank $r$ on $E$.
	Fix an integer $i$ with $1\le i\le r-2$, and set $N=T_{i+2}(M)$.
	Section~\ref{sec:geometry} proves \eqref{ineq:weight} when $N$ is realizable over $\C$.
	We now extend the proof to arbitrary $M$ using the Chow ring.
	
	\subsection{The Chow ring and the Hodge--Riemann relations}
	\label{subsec:chow}
	The Chow ring of a matroid was introduced by Feichtner and Yuzvinsky~\cite{FY04}.
	Adiprasito, Huh, and Katz~\cite{AHK18} proved that it satisfies Poincar\'e duality, the hard Lefschetz theorem, and the Hodge--Riemann relations.
	Braden, Huh, Matherne, Proudfoot, and Wang~\cite{BHMPW22} subsequently obtained new proofs of these results using a semi-small decomposition of the Chow ring.
	We use the presentation in~\cite{BHMPW22},
	with coefficients in $\R$.
	Let $\F^\circ(M)=\mathcal{L}(M)\setminus\{\varnothing,E\}$.
	\begin{definition}\label{def:Chow}
		Let $M$ be a loopless matroid.
		The \emph{Chow ring} of $M$ is the quotient ring
		\begin{equation*}
			A^*(M)=
			\R[x_F:F\in\F^\circ(M)]\,/\,(I_M+J_M),
		\end{equation*}
		where
		\begin{align*}
			I_M &=
			\left\langle
			x_{F_1}x_{F_2}:F_1,F_2\in\F^\circ(M)
			\text{ are incomparable}
			\right\rangle,\\
			J_M &=
			\left\langle
			\sum_{F\ni e}x_F-\sum_{F\ni f}x_F
			:
			e,f\in E
			\right\rangle.
		\end{align*}
	\end{definition}
	
	Each variable $x_F$ has degree $1$.
	Thus $A^*(M)=\bigoplus_{i\ge0}A^i(M)$ is a graded algebra, with $A^i(M)$ its component of degree $i$; see~\cite{AHK18,BHMPW22}.
	For $e\in E$, define $\alpha_M,\beta_M\in A^1(M)$ by
	\begin{equation*}
		\alpha_M
		=
		\sum_{\substack{F\in\F^\circ(M)\\e\in F}}x_F,
		\qquad
		\beta_M
		=
		\sum_{\substack{F\in\F^\circ(M)\\e\notin F}}x_F.
	\end{equation*}
	The relations in $J_M$ show that $\alpha_M$ and $\beta_M$ do not depend on $e$.
	
	\begin{proposition}[{\cite{AHK18}}]\label{prop:chow-degree}
		Let $M$ be a loopless matroid of rank $r\ge1$.
		\begin{enumerate}
			\item [{\rm (i)}]
			$A^0(M)=\R$ and $A^i(M)=0$ for $i>r-1$.
			
			\item [{\rm (ii)}]
			For every complete flag $\varnothing\subsetneq F_1\subsetneq\cdots \subsetneq F_{r-1}\subsetneq E$, we have
			\begin{equation*}
				x_{F_1}\cdots x_{F_{r-1}}=\alpha_M^{r-1}\ne0.
			\end{equation*}
			These products span $A^{r-1}(M)$, so $A^{r-1}(M)=\{t\alpha_M^{r-1}:t\in\R\}$.
		\end{enumerate}
	\end{proposition}
	
	Define the linear map $\degr_M:A^*(M)\to\R$ by
	\begin{equation*}
		\degr_M(t\alpha_M^{r-1})=t
		\quad\text{for all }t\in\R,
		\qquad
		\degr_M|_{A^i(M)}=0
		\quad\text{for }i\ne r-1.
	\end{equation*}
	Proposition~\ref{prop:chow-degree} shows that this map is well defined and restricts to an isomorphism on $A^{r-1}(M)$.
	
	\begin{proposition}[Poincar\'e duality {\cite{AHK18}}]
		\label{prop:chow-poincare-duality}
		Let $M$ be a loopless matroid of rank $r\ge1$.
		For every $0\le i\le r-1$, the pairing
		\[
		A^i(M)\times A^{r-1-i}(M)\longrightarrow\R,
		\qquad
		(u,v)\longmapsto\degr_M(uv)
		\]
		is nondegenerate.
		In particular, if $u\in A^i(M)$ satisfies $\degr_M(uv)=0$ for every $v\in A^{r-1-i}(M)$, then $u=0$.
	\end{proposition}
	
	The coefficients of the reduced characteristic polynomial are given by the following degree formula.
	
	\begin{proposition}[{\cite{AHK18}}]
		\label{prop:degree-coeff}
		Let $M$ be a loopless matroid of rank $r\ge1$.
		For $0\le i\le r-1$, we have
		\begin{equation}\label{eq:degree-coeff}
			\degr_M\left(\alpha_M^{r-1-i}\beta_M^i\right)=\nu_i(M).
		\end{equation}
	\end{proposition}
	
	A function $c:2^E\to\R$ with $c(\varnothing)=c(E)=0$ is \emph{submodular} if
	\begin{equation*}
		c(S)+c(T)\ge c(S\cap T)+c(S\cup T)
		\quad\text{for all }S,T\subseteq E.
	\end{equation*}
	It is \emph{strictly submodular} if the inequality is strict whenever $S$ and $T$ are incomparable.
	For such a function $c$, define
	\begin{equation}\label{l(c)}
		\ell(c)=\sum_{F\in\F^\circ(M)}c(F)x_F\in A^1(M).
	\end{equation}
	The following are the degree zero and degree one cases of the Hodge--Riemann relations.
	\begin{thm}[{\cite{AHK18}}]
		\label{thm:chow-hr}
		Let $M$ be a loopless matroid of rank $r\ge3$, and let $c$ be strictly submodular.
		Then
		\begin{equation*}
			\degr_M\left(\ell(c)^{r-1}\right)>0.
		\end{equation*}
		For every nonzero $u\in A^1(M)$ satisfying $\ell(c)^{r-2}u=0$, we have
		\begin{equation*}
			\degr_M\left(\ell(c)^{r-3}u^2\right)<0.
		\end{equation*}
	\end{thm}
	
	Let $V$ be a real vector space of dimension $m$, and let $B:V\times V\to\R$ be a symmetric bilinear form.
	For a basis $e_1,\ldots,e_m$ of $V$, the matrix of $B$ is $(B(e_i,e_j))_{i,j=1}^{m}$.
	By Sylvester's law of inertia~\cite[Theorem~4.5.8]{HJ13}, the numbers of positive, negative, and zero eigenvalues of this matrix are independent of the chosen basis.
	
	We use the following special case of the Hodge--Riemann forms of Adiprasito, Huh, and Katz~\cite[Definition~7.3]{AHK18}.
	Let $M$ be a loopless matroid of rank $k+2$, where $k\ge1$.
	Consider the symmetric bilinear form
	\[
	B_k:A^1(M)\times A^1(M)\longrightarrow\R
	\]
	defined by
	\begin{equation}\label{Bbeta}
		B_k(u,v)=\degr_M(\beta_M^{k-1}uv).
	\end{equation}
	The next result follows from~\cite[Lemmas~9.6 and~9.7]{AHK18}, or by approximation from the Lorentzian property of the matroid volume polynomial; see~\cite[Section~4]{BL26}.
	For completeness, we give a proof using Theorem~\ref{thm:chow-hr}.
	
	\begin{corollary}\label{cor:onepositive}
		Let $M$ be a loopless matroid of rank $k+2$, where $k\ge1$, and let $B_k$ be the form defined in~\eqref{Bbeta}.
		Set $m=\dim_\R A^1(M)$.
		For any basis $e_1,\ldots,e_m$ of $A^1(M)$, the matrix $(B_k(e_i,e_j))_{i,j=1}^m$ has exactly one positive eigenvalue.
		Moreover, $B_k(\alpha_M,\alpha_M)=\nu_{k-1}(M)>0$.
	\end{corollary}
	
	\begin{proof}
		Let $E$ be the ground set of $M$.
		For $\varepsilon>0$, define $c_\varepsilon:2^E\to\R$ by
		\begin{equation*}
			c_\varepsilon(S)=
			\begin{cases}
				0,
				& S=\varnothing,\\[1mm]
				1-\dfrac{|S|}{|E|}
				+\varepsilon\, |S|(|E|-|S|),
				& \varnothing\subsetneq S\subseteq E.
			\end{cases}
		\end{equation*}
		For incomparable subsets $S,T\subseteq E$, we have
		\begin{align*}
			c_\varepsilon(S)+c_\varepsilon(T)
			-c_\varepsilon(S\cap T)-c_\varepsilon(S\cup T)&=
			2\,\varepsilon \,|S\setminus T|\,|T\setminus S|
			+\begin{cases}
				1,&S\cap T=\varnothing\\
				0,&S\cap T\ne\varnothing
			\end{cases}\\
			&\ge
			2\,\varepsilon\, |S\setminus T|\,|T\setminus S|\\
			&>0.
		\end{align*}
		Thus $c_\varepsilon$ is strictly submodular.
		
		Averaging the defining expression for $\beta_M$ over $e\in E$ gives
		\begin{equation*}
			\beta_M
			=
			\frac1{|E|}
			\sum_{e\in E}
			\sum_{\substack{F\in\F^\circ(M)\\e\notin F}}x_F
			=
			\sum_{F\in\F^\circ(M)}
			\left(1-\frac{|F|}{|E|}\right)x_F.
		\end{equation*}
		By \eqref{l(c)}, we have
		\begin{equation*}
			\ell(c_\varepsilon)=\sum_{F\in\F^\circ(M)}c_\varepsilon(F)x_F
			=
			\beta_M+
			\varepsilon\,\sum_{F\in\F^\circ(M)}
			|F|(|E|-|F|)x_F.
		\end{equation*}
		Thus
		\begin{equation}\label{lep}
			\ell(c_\varepsilon)\to\beta_M\quad \text{as } \varepsilon\to0^+.
		\end{equation}
		
		By Theorem~\ref{thm:chow-hr}, $\degr_M(\ell(c_\varepsilon)^{k+1})>0$.
		For $u\in A^1(M)$, define
		\begin{equation}\label{def-u}
			u_\varepsilon=u-\frac{\degr_M\left(\ell(c_\varepsilon)^{k}u\right)}
			{\degr_M\left(\ell(c_\varepsilon)^{k+1}\right)}\ell(c_\varepsilon).
		\end{equation}
		Then
		\begin{equation*}
			\degr_M\left(\ell(c_\varepsilon)^{k} u_\varepsilon\right)
			=
			\degr_M\left(\ell(c_\varepsilon)^{k} u\right)-
			\frac{\degr_M\left(\ell(c_\varepsilon)^{k} u\right)}
			{\degr_M\left(\ell(c_\varepsilon)^{k+1}\right)}
			\degr_M\left(\ell(c_\varepsilon)^{k+1}\right)
			=0.
		\end{equation*}
		Since $\ell(c_\varepsilon)^{k} u_\varepsilon\in A^{k+1}(M)$
		and the degree map is an isomorphism on $A^{k+1}(M)$, we have
		\begin{equation*}
			\ell(c_\varepsilon)^{k} u_\varepsilon=0.
		\end{equation*}
		Theorem~\ref{thm:chow-hr} and \eqref{def-u} therefore give
		\begin{equation*}
			\degr_M(\ell(c_\varepsilon)^{k-1}u_\varepsilon^2)=
			\degr_M(\ell(c_\varepsilon)^{k-1}u^2)
			-\frac{\degr_M(\ell(c_\varepsilon)^{k} u)^2}
			{\degr_M(\ell(c_\varepsilon)^{k+1})}
			\le 0.
		\end{equation*}
		Multiplying by $\degr_M\left(\ell(c_\varepsilon)^{k+1}\right)>0$ and passing to the limit in~\eqref{lep}, we obtain
		\begin{equation*}
			\degr_M(\beta_M^{k-1}u^2)\degr_M(\beta_M^{k+1})
			\le \degr_M(\beta_M^{k} u)^2.
		\end{equation*}
		Equivalently, by~\eqref{Bbeta},
		\begin{equation}\label{BB<B}
			B_k(u,u)B_k(\beta_M,\beta_M)
			\le B_k(u,\beta_M)^2.
		\end{equation}
		
		Let
		\begin{equation*}
			U=\{u\in A^1(M):B_k(u,\beta_M)=0\}.
		\end{equation*}
		By Proposition~\ref{prop:degree-coeff}, we have
		\begin{equation*}
			B_k(\beta_M,\beta_M)
			=\degr_M(\beta_M^{k+1})
			=\nu_{k+1}(M)>0.
		\end{equation*}
		It follows from~\eqref{BB<B} that $B_k(u,u)\le0$ for every $u\in U$, so $B_k$ is negative semidefinite on $U$.
		
		Consider the linear functional
		\begin{equation*}
			f:A^1(M)\longrightarrow\R,\qquad
			f(u)=B_k(u,\beta_M).
		\end{equation*}
		Then $U=\ker f$.
		Since $B_k(\beta_M,\beta_M)>0$, the functional $f$ is nonzero.
		Hence $\dim_\R\operatorname{im}f=1$ and
		\begin{equation*}
			\dim_\R U=\dim_\R \ker f
			=\dim_\R A^1(M)-\dim_\R\operatorname{im}f
			=\dim_\R A^1(M)-1.
		\end{equation*}
		
		Fix a basis $e_1,\ldots,e_m$ of $A^1(M)$.
		Suppose that the matrix $(B_k(e_i,e_j))_{i,j=1}^m$ has at least two positive eigenvalues.
		By the spectral theorem~\cite[Theorem~4.1.5]{HJ13}, there exists a two-dimensional subspace $V\subseteq A^1(M)$ such that $B_k(v,v)>0$ for every nonzero $v\in V$.
		The dimension formula for subspaces~\cite[p.~4]{HJ13} gives
		\begin{equation*}
			\dim_\R(V\cap U)
			\ge \dim_\R V+\dim_\R U-\dim_\R A^1(M)
			=1.
		\end{equation*}
		Let $0\ne v\in V\cap U$.
		Then $v\in V$ gives $B_k(v,v)>0$, whereas $v\in U$ gives $B_k(v,v)\le0$, a contradiction.
		Thus the matrix $(B_k(e_i,e_j))_{i,j=1}^m$ has at most one positive eigenvalue.
		Since $B_k(\beta_M,\beta_M)>0$, it has exactly one positive eigenvalue.
		
		Finally, by Proposition~\ref{prop:degree-coeff}, we have
		\begin{equation*}
			B_k(\alpha_M,\alpha_M)
			=\degr_M(\beta_M^{k-1}\alpha_M^2)
			=\nu_{k-1}(M)>0.
		\end{equation*}
		This completes the proof.
	\end{proof}
	
	\subsection{Boundary classes and intersection numbers}
	\label{subsec:chow-boundary}
	
	In this subsection, all tensor products are over $\R$.
	We recall the pullback and pushforward maps of~\cite[Propositions~2.24 and~2.25]{BHMPW22}.
	For $F\in\F^\circ(M)$, the \emph{pullback} is the graded algebra homomorphism
	\begin{equation*}
		\varphi_F:A^*(M)\longrightarrow
		A^*(M|F)\otimes A^*(M/F)
	\end{equation*}
	defined on generators by
	\begin{equation}\label{eq:pullback}
		\varphi_F(x_G)=
		\begin{cases}
			x_G\otimes1,
			& G\subsetneq F,\\[1mm]
			1\otimes x_{G\setminus F},
			& F\subsetneq G,\\[1mm]
			-\alpha_{M|F}\otimes1-1\otimes\beta_{M/F},
			& G=F,\\[1mm]
			0,
			& G\text{ and }F\text{ are incomparable},
		\end{cases}
	\end{equation}
	for $G\in\F^\circ(M)$.
	By~\cite[Proposition~2.24]{BHMPW22},
	\begin{equation}\label{eq:pullback-ab}
		\varphi_F(\alpha_M)=1\otimes\alpha_{M/F},
		\qquad
		\varphi_F(\beta_M)=\beta_{M|F}\otimes1.
	\end{equation}
	The \emph{pushforward} is the linear map
	\begin{equation*}
		\psi_F:A^*(M|F)\otimes A^*(M/F)
		\longrightarrow A^*(M)
	\end{equation*}
	defined on monomials by
	\begin{equation*}
		\psi_F\left(
		\prod_{j=1}^{b}x_{H_j}
		\otimes
		\prod_{i=1}^{a}x_{G_i\setminus F}
		\right)
		=
		x_F\prod_{i=1}^{a}x_{G_i}\prod_{j=1}^{b}x_{H_j}\quad \text{for all }a,b\in\Z_{\ge0},
	\end{equation*}
	where $G_i,H_j\in\mathcal{L}(M)$ satisfy $F\subsetneq G_i\subsetneq E$ and $\varnothing\subsetneq H_j\subsetneq F$.
	Empty products are understood to be $1$.
	In particular, $\psi_F(1\otimes1)=x_F$.
	The map $\psi_F$ increases degree by one.
	
	The degree maps of $M|F$ and $M/F$ induce a linear map
	\begin{equation*}
		\degr_{M|F}\otimes\degr_{M/F}:
		A^*(M|F)\otimes A^*(M/F)\longrightarrow\R
	\end{equation*}
	by
	\begin{equation*}
		(\degr_{M|F}\otimes\degr_{M/F})(u\otimes v)
		=
		\degr_{M|F}(u)\degr_{M/F}(v)
	\end{equation*}
	for $u\in A^*(M|F)$ and $v\in A^*(M/F)$.
	
	\begin{proposition}[{\cite[Proposition~2.25 and Remark~2.26]{BHMPW22}}]\label{prop:chow-maps}
		Let $F\in\F^\circ(M)$.
		For every $z\in A^*(M)$, $u\in A^*(M|F)$ and $v\in A^*(M/F)$, we have
		\begin{equation}\label{eq:projection}
			z\,\psi_F(u\otimes v)
			=
			\psi_F\bigl(\varphi_F(z)(u\otimes v)\bigr)
		\end{equation}
		and
		\begin{equation}\label{eq:push-degree}
			\degr_M\bigl(\psi_F(u\otimes v)\bigr)
			=
			\degr_{M|F}(u)\degr_{M/F}(v).
		\end{equation}
	\end{proposition}
	
	We will use the following consequence of Proposition~\ref{prop:chow-maps}.
	
	\begin{corollary}\label{cor:local-degree}
		Let $F\in\F^\circ(M)$.
		For every $z\in A^{r-2}(M)$, we have
		\begin{equation*}
			\degr_M(z\, x_F)
			=
			(\degr_{M|F}\otimes\degr_{M/F})
			\bigl(\varphi_F(z)\bigr).
		\end{equation*}
	\end{corollary}
	
	\begin{proof}
		Taking $u=1\in A^*(M|F)$ and $v=1\in A^*(M/F)$ in~\eqref{eq:projection} gives
		\begin{equation*}
			z\,\psi_F(1\otimes1)
			=
			\psi_F\bigl(\varphi_F(z)(1\otimes1)\bigr).
		\end{equation*}
		Since $\psi_F(1\otimes1)=x_F$, this becomes
		\begin{equation*}
			z\,x_F
			=
			\psi_F\bigl(\varphi_F(z)(1\otimes1)\bigr)
			=
			\psi_F\bigl(\varphi_F(z)\bigr).
		\end{equation*}
		Write
		\begin{equation*}
			\varphi_F(z)=\sum_{j=1}^{s}u_j\otimes v_j,
		\end{equation*}
		where $u_j\in A^*(M|F)$ and $v_j\in A^*(M/F)$.
		By linearity and~\eqref{eq:push-degree},
		\begin{equation*}
			\begin{aligned}
				\degr_M(zx_F)
				&=\degr_M\bigl(\psi_F(\varphi_F(z))\bigr)\\
				&=\sum_{j=1}^{s}
				\degr_M\bigl(\psi_F(u_j\otimes v_j)\bigr)\\
				&=\sum_{j=1}^{s}
				\degr_{M|F}(u_j)\degr_{M/F}(v_j)\\
				&=(\degr_{M|F}\otimes\degr_{M/F})
				\bigl(\varphi_F(z)\bigr).
			\end{aligned}
		\end{equation*}
		This completes the proof.
	\end{proof}
	
	We now apply the form in~\eqref{Bbeta} to $N$. Since $N$ has rank $i+2$, this gives
	\[
	B_i(u,v)=\degr_N(\beta_N^{i-1}uv),
	\qquad u,v\in A^1(N).
	\]
	
	By the definition of truncation, $N|G=M|G$ for $G\in\F_i(M)\cup\F_{i+1}(M)$.
	For $F\in\F_i(M)$, the flats of rank one in $N/F$ and $M/F$ are
	the sets $H\setminus F$ with $F\subsetneq H\in\F_{i+1}(M)$.
	
	\begin{proposition}\label{prop:intersection}
		For $F\in\F_i(M)$ and $H\in\F_{i+1}(M)$, we have
		\begin{itemize}
			\item[\rm(i)] $B_i(\alpha_N,x_F)=\nu_{i-1}(M|F)$,
			\item[\rm(ii)] $B_i(\alpha_N,x_H)=0,$
			\item[\rm(iii)] $B_i(x_F,x_F)=-\bigl(|\F_1(M/F)|-1\bigr)\,\nu_{i-1}(M|F)$,
			\item[\rm(iv)] $B_i(x_H,x_H)=-\nu_{i-1}(M|H).$
		\end{itemize}
		Moreover,
		\begin{equation}\label{brfh}
			B_i(x_F,x_H)=
			\begin{cases}
				\nu_{i-1}(M|F),&F\subseteq H,\\
				0,&F\not\subseteq H.
			\end{cases}
		\end{equation}
		For distinct flats $G,G'\in\F^\circ(N)$ of the same rank, we have $B_i(x_G,x_{G'})=0$.
	\end{proposition}
	
	\begin{proof}
		For $z\in A^1(N)$, we have $\beta_N^{i-1}z\in A^{i}(N)$.
		By Corollary~\ref{cor:local-degree},
		for $G\in\F^\circ(N)$,
		\begin{align}\label{eq:chow-boundary-pairing}
			B_i(z,x_G)&=\degr_N(\beta_N^{i-1}z\,x_G)\nonumber\\
			&=(\degr_{N|G}\otimes\degr_{N/G})\left(\varphi_G(\beta_N^{i-1}z)\right)\nonumber\\
			&=(\degr_{N|G}\otimes\degr_{N/G})\bigl(\varphi_G(\beta_N)^{i-1}\,\varphi_G(z)\bigr)\nonumber\\
			&=(\degr_{N|G}\otimes\degr_{N/G})
			\left((\beta_{N|G}^{i-1}\otimes1)\,\varphi_G(z)\right),
		\end{align}
		where the third equality follows because $\varphi_G$ is an algebra homomorphism, and the last follows from~\eqref{eq:pullback-ab}.
		
		For \textup{(i)} and \textup{(iii)}, let $F\in\F_i(M)$.
		Then $N/F$ and $N|F$ have ranks $2$ and $i$, respectively.
		Proposition~\ref{prop:chow-degree}\textup{(i)} gives $A^{i}(N|F)=0$.
		Hence
		\begin{equation*}
			\alpha_{N|F}\beta_{N|F}^{i-1}=0.
		\end{equation*}
		By Proposition~\ref{prop:degree-coeff},
		\[
		\degr_{N/F}(\alpha_{N/F})
		=\nu_0(N/F)=1,\quad
		\degr_{N/F}(\beta_{N/F})
		=\nu_1(N/F),\quad
		\degr_{N|F}(\beta_{N|F}^{i-1})
		=\nu_{i-1}(M|F).
		\]
		Also, \eqref{eq:w-nu} gives $\nu_1(N/F)=W_1(N/F)-1=|\F_1(M/F)|-1$.
		By~\eqref{eq:pullback-ab} and \eqref{eq:pullback},
		\begin{equation*}
			\varphi_F(\alpha_N)=1\otimes\alpha_{N/F},
			\qquad
			\varphi_F(x_F)
			=-\alpha_{N|F}\otimes1-1\otimes\beta_{N/F}.
		\end{equation*}
		Taking $z=\alpha_N$ and $G=F$ in~\eqref{eq:chow-boundary-pairing} yields
		\begin{align*}
			B_i(\alpha_N,x_F)
			&=(\degr_{N|F}\otimes\degr_{N/F})
			\left((\beta_{N|F}^{i-1}\otimes1)
			(1\otimes\alpha_{N/F})\right)\\
			&=(\degr_{N|F}\otimes\degr_{N/F})
			\left(\beta_{N|F}^{i-1}\otimes\alpha_{N/F}\right)\\
			&=\degr_{N/F}(\alpha_{N/F})\,
			\degr_{N|F}(\beta_{N|F}^{i-1})=\nu_{i-1}(M|F).
		\end{align*}
		Similarly, taking $z=x_F$ and $G=F$ in~\eqref{eq:chow-boundary-pairing} and using $\alpha_{N|F}\beta_{N|F}^{i-1}=0$ yields
		\begin{align*}
			B_i(x_F,x_F)
			&=(\degr_{N|F}\otimes\degr_{N/F})
			\left((\beta_{N|F}^{i-1}\otimes1)
			(-\alpha_{N|F}\otimes1-1\otimes\beta_{N/F})\right)\\
			&=(\degr_{N|F}\otimes\degr_{N/F})
			\left(-\beta_{N|F}^{i-1}\otimes\beta_{N/F}
			-\alpha_{N|F}\beta_{N|F}^{i-1}\otimes1\right)\\
			&=-(\degr_{N|F}\otimes\degr_{N/F})
			\left(\beta_{N|F}^{i-1}\otimes\beta_{N/F}\right)\\
			&=-\degr_{N/F}(\beta_{N/F})\,
			\degr_{N|F}(\beta_{N|F}^{i-1})\\
			&=-\bigl(|\F_1(M/F)|-1\bigr)\,\nu_{i-1}(M|F).
		\end{align*}
		This proves \textup{(i)} and \textup{(iii)}.
		
		For \textup{(ii)} and \textup{(iv)}, let $H\in\F_{i+1}(M)$.
		Since $N/H$ has rank $1$, Proposition~\ref{prop:chow-degree} gives $A^*(N/H)=\R$ and $\alpha_{N/H}=\beta_{N/H}=0$.
		Moreover, Proposition~\ref{prop:degree-coeff} gives
		\begin{equation*}
			\degr_{N/H}(1)=\nu_0(N/H)=1.
		\end{equation*}
		By~\eqref{eq:pullback-ab} and~\eqref{eq:pullback},
		\begin{equation*}
			\varphi_H(\alpha_N)=0,
			\qquad
			\varphi_H(x_H)=-\alpha_{N|H}\otimes1.
		\end{equation*}
		Taking $z=\alpha_N$ and $G=H$ in~\eqref{eq:chow-boundary-pairing} gives
		\begin{align*}
			B_i(\alpha_N,x_H)
			=(\degr_{N|H}\otimes\degr_{N/H})\left((\beta_{N|H}^{i-1}\otimes1) \varphi_H(\alpha_N)\right)=0.
		\end{align*}
		This proves \textup{(ii)}.
		Similarly, using $\varphi_H(x_H)=-\alpha_{N|H}\otimes1$, we obtain
		\begin{align*}
			B_i(x_H,x_H)
			&=(\degr_{N|H}\otimes\degr_{N/H})
			\left((\beta_{N|H}^{i-1}\otimes1)
			(-\alpha_{N|H}\otimes1)\right)\\
			&=-(\degr_{N|H}\otimes\degr_{N/H})
			\left(\alpha_{N|H}\beta_{N|H}^{i-1}\otimes1\right)\\
			&=-\degr_{N/H}(1)\,
			\degr_{N|H}(\alpha_{N|H}\beta_{N|H}^{i-1})=-\nu_{i-1}(M|H).
		\end{align*}
		The last equality follows from $\degr_{N/H}(1)=1$ and Proposition~\ref{prop:degree-coeff}, applied to $N|H$, which has rank $i+1$.
		This proves \textup{(iv)}.
		
		It remains to compute $B_i(x_F,x_H)$ for $F\in\F_i(M)$ and $H\in\F_{i+1}(M)$.
		If $F\not\subseteq H$, then $F$ and $H$ are incomparable.
		Definition~\ref{def:Chow} gives $x_Fx_H=0$.
		Hence $B_i(x_F,x_H)=\degr_N(\beta_N^{i-1}x_F x_H)=0$.
		Suppose that $F\subseteq H$.
		Then $H\setminus F$ has rank $1$ in the rank $2$ matroid $N/F$, so
		$\varnothing\subsetneq H\setminus F\subsetneq E\setminus F$
		is a complete flag in $\mathcal{L}(N/F)$.
		By Proposition~\ref{prop:chow-degree}\textup{(ii)} and Proposition~\ref{prop:degree-coeff},
		\[
		\degr_{N/F}(x_{H\setminus F})
		=\degr_{N/F}(\alpha_{N/F})
		=1.
		\]
		Since $\varphi_F(x_H)=1\otimes x_{H\setminus F}$,
		taking $z=x_H$ and $G=F$ in~\eqref{eq:chow-boundary-pairing} gives
		\begin{align*}
			B_i(x_F,x_H)
			&=B_i(x_H,x_F)\\
			&=(\degr_{N|F}\otimes\degr_{N/F})
			\left((\beta_{N|F}^{i-1}\otimes1)
			\varphi_F(x_H)\right)\\
			&=(\degr_{N|F}\otimes\degr_{N/F})
			\left((\beta_{N|F}^{i-1}\otimes1)
			(1\otimes x_{H\setminus F})\right)\\
			&=(\degr_{N|F}\otimes\degr_{N/F})
			\left(\beta_{N|F}^{i-1}\otimes x_{H\setminus F}\right)\\
			&=\degr_{N/F}(x_{H\setminus F})\,
			\degr_{N|F}(\beta_{N|F}^{i-1})=\nu_{i-1}(M|F).
		\end{align*}
		This proves~\eqref{brfh}.
		
		Finally, let $G,G'\in\F^\circ(N)$ be distinct flats
		of the same rank.
		Then $G$ and $G'$ are incomparable, so $x_Gx_{G'}=0$.
		Consequently,
		\begin{equation*}
			B_i(x_G,x_{G'})
			=\degr_N(\beta_N^{i-1}x_Gx_{G'})
			=0.
		\end{equation*}
		This completes the proof.
	\end{proof}
	
	\subsection{The descended algebra and numerical pullback}
	\label{subsec:chow-descent}
	
	The \emph{annihilator} of $\beta_N^{i-1}$ is the homogeneous ideal
	\[
	\Ann(\beta_N^{i-1})
	=\{z\in A^*(N)\mid z\beta_N^{i-1}=0\}.
	\]
	Define the graded algebra
	\[
	R_i=A^*(N)/\Ann(\beta_N^{i-1}),
	\]
	and let $R_i^j$ denote its component of degree $j$.
	Since $N$ has rank $i+2$, for $j>2$ and $z\in A^j(N)$ we have $\beta_N^{i-1}z\in A^{i-1+j}(N)=0$.
	Hence $R_i^j=0$ for $j>2$.
	
	By~\eqref{eq:degree-coeff} and~\eqref{eq:trun-mobi},
	\begin{equation}\label{eq:chow-alpha-square}
		B_i(\alpha_N,\alpha_N)
		=\degr_N(\alpha_N^2\beta_N^{i-1})
		=\nu_{i-1}(N)>0.
	\end{equation}
	In particular, $\beta_N^{i-1}\ne0$.
	Since $A^0(N)=\R$, we have $\Ann(\beta_N^{i-1})\cap A^0(N)=0$ and hence $R_i^0=\R$.
	
	The kernel of multiplication by $\beta_N^{i-1}$ on $A^2(N)$ is $\Ann(\beta_N^{i-1})\cap A^2(N)$.
	Thus multiplication by $\beta_N^{i-1}$ induces an injective linear map
	\[
	R_i^2\longrightarrow A^{i+1}(N),\qquad
	\overline z\longmapsto\beta_N^{i-1}z,
	\]
	where $\overline z$ denotes the class of $z\in A^2(N)$ in $R_i^2$.
	By~\eqref{eq:chow-alpha-square}, its image contains the nonzero element $\beta_N^{i-1}\alpha_N^2$.
	Since $A^{i+1}(N)$ is one-dimensional, this map is an isomorphism.
	Define the degree map $\degr_{R_i}:R_i^2\to\R$ by
	\[
	\degr_{R_i}(\overline z)=\degr_N(\beta_N^{i-1}z)
	\qquad\text{for every }z\in A^2(N).
	\]
	For all $u,v\in A^1(N)$, we have
	\[
	\degr_{R_i}(\overline u\,\overline v)
	=\degr_N(\beta_N^{i-1}uv)
	=B_i(u,v).
	\]
	
	\begin{proposition}\label{prop:chow-quotient-pairing}
		For any basis $\overline u_1,\ldots,\overline u_m$ of $R_i^1$,
		the symmetric matrix
		\[
		\bigl(\degr_{R_i}(\overline u_j\,\overline u_k)\bigr)_{1\le j,k\le m}
		\]
		has exactly one positive eigenvalue.
	\end{proposition}
	
	\begin{proof}
		For each $1\le j\le m$, choose a representative $u_j\in A^1(N)$ of $\overline u_j$.
		Recall that
		\[
		R_i^1=A^1(N)/\bigl(\Ann(\beta_N^{i-1})\cap A^1(N)\bigr).
		\]
		Let $w_1,\ldots,w_s$ be a basis of $\Ann(\beta_N^{i-1})\cap A^1(N)$.
		Then $u_1,\ldots,u_m,w_1,\ldots,w_s$ form a basis of $A^1(N)$.
		
		Since $\beta_N^{i-1}w_\ell=0$ for $1\le\ell\le s$, we have
		\[
		B_i(w_\ell,v)=\degr_N(\beta_N^{i-1}w_\ell v)=0
		\qquad\text{for every }v\in A^1(N).
		\]
		For $1\le j,k\le m$, we have
		\[
		B_i(u_j,u_k)=\degr_{R_i}(\overline u_j\,\overline u_k).
		\]
		Thus the matrix of $B_i$ in the basis $u_1,\ldots,u_m,w_1,\ldots,w_s$ is
		\[
		\begin{pmatrix}
			\bigl(\degr_{R_i}(\overline u_j\,\overline u_k)\bigr)_{1\le j,k\le m} & 0\\
			0 & 0
		\end{pmatrix}.
		\]
		Corollary~\ref{cor:onepositive} shows that this block matrix has exactly one positive eigenvalue.
		The same is therefore true of its upper-left block.
	\end{proof}
	
	Let $\overline{B}_i$ denote the induced bilinear form on $R_i^1$.
	For all $u,v\in A^1(N)$, we have
	\[
	\overline{B}_i(\overline u,\overline v)
	=\degr_{R_i}(\overline u\,\overline v)
	=B_i(u,v).
	\]
	Define the subspaces
	\[
	\begin{aligned}
		\mathcal E_i
		&=\Span\{\overline{x_H}\mid H\in\F_{i+1}(M)\},\\
		\mathcal D_i
		&=\{\overline u\in R_i^1\mid
		\overline{B}_i(\overline u,\overline{x_H})=0
		\text{ for every }H\in\F_{i+1}(M)\}.
	\end{aligned}
	\]
	By Proposition~\ref{prop:intersection}, the classes $\overline{x_H}$ are pairwise orthogonal, and
	\[
	\overline{B}_i(\overline{x_H},\overline{x_H})
	=-\nu_{i-1}(M|H)<0
	\qquad\text{for every }H\in\F_{i+1}(M).
	\]
	Thus $\overline{B}_i$ is negative definite on $\mathcal E_i$.
	Since $\mathcal D_i$ is the orthogonal complement of $\mathcal E_i$, we have the orthogonal direct sum
	\[
	R_i^1=\mathcal D_i\oplus\mathcal E_i.
	\]
	Also, $\overline{\alpha_N}\in\mathcal D_i$, since $\overline{B}_i(\overline{\alpha_N},\overline{x_H})=0$ for every $H\in\F_{i+1}(M)$.
	
	\begin{definition}\label{def:chow-descended-algebra}
		The \emph{descended algebra} $\mathcal A_i(M)$ is the graded subalgebra
		\[
		\mathcal A_i(M)=R_i^0\oplus\mathcal D_i\oplus R_i^2
		\]
		of $R_i$.
	\end{definition}
	
	This is a subalgebra because $R_i^0=\R$, the product of two elements of $\mathcal D_i$ lies in $R_i^2$, and $R_i^j=0$ for $j>2$.
	Assign a real weight $a_F$ to each $F\in\F_i(M)$.
	Since $\nu_{i-1}(M|F)>0$ for every $F\in\F_i(M)$, we may define
	\[
	v(a)=\sum_{F\in\F_i(M)}
	\frac{a_F}{\nu_{i-1}(M|F)}\,\overline{x_F}\in R_i^1.
	\]
	The orthogonal decomposition $R_i^1=\mathcal D_i\oplus\mathcal E_i$ determines a unique element $v^\#(a)\in\mathcal D_i$ such that
	\[
	v^\#(a)-v(a)\in\mathcal E_i.
	\]
	We call $v^\#(a)$ the \emph{numerical pullback} of $v(a)$.
	
	\begin{proposition}\label{prop:chow-numerical-pullback}
		For any real weights $(a_F)_{F\in\F_i(M)}$, we have
		\begin{equation}\label{eq:chow-numerical-pullback}
			v^\#(a)
			=
			v(a)+
			\sum_{H\in\F_{i+1}(M)}
			\frac{\displaystyle\sum_{F\in\F_i(M|H)}a_F}
			{\nu_{i-1}(M|H)}\,\overline{x_H}.
		\end{equation}
	\end{proposition}
	
	\begin{proof}
		Since $v^\#(a)-v(a)\in\mathcal E_i$, write
		\[
		v^\#(a)=v(a)+\sum_{H\in\F_{i+1}(M)}c_H\,\overline{x_H}.
		\]
		For each $H\in\F_{i+1}(M)$, Proposition~\ref{prop:intersection} gives
		\begin{equation}\label{eq:weighted-cross}
			\overline{B}_i(v(a),\overline{x_H})
			=\sum_{F\in\F_i(M)}
			\frac{a_F}{\nu_{i-1}(M|F)}
			B_i(x_F,x_H)=\sum_{F\in\F_i(M|H)}a_F.
		\end{equation}
		Since $v^\#(a)\in\mathcal D_i$ and the classes $\overline{x_H}$ are pairwise orthogonal, we have
		\[
		0=\overline{B}_i(v^\#(a),\overline{x_H})=\overline{B}_i(v(a),\overline{x_H})
		+c_H\,\overline{B}_i(\overline{x_H},\overline{x_H})=\sum_{F\in\F_i(M|H)}a_F-c_H\nu_{i-1}(M|H).
		\]
		Thus
		\[
		c_H=
		\frac{\displaystyle\sum_{F\in\F_i(M|H)}a_F}
		{\nu_{i-1}(M|H)},
		\]
		which proves \eqref{eq:chow-numerical-pullback}.
	\end{proof}
	
	When $N$ is realizable over $\C$, the coefficient of $\overline{x_H}$ in~\eqref{eq:chow-numerical-pullback} agrees with that of $E_H$ in the geometric formula~\eqref{eq:geometry-numerical-pullback} for every $H\in\F_{i+1}(M)$.
	
	\subsection{Proof of the weighted inequality}
	\label{subsec:chow-weighted-proof}
	
	\begin{proof}[Proof of Theorem~\ref{thm:weighted-incidence}]
		If $i=r-1$, then $\F_r(M)=\{E\}$ and $|\F_1(M/F)|=1$ for every $F\in\F_{r-1}(M)$, so~\eqref{ineq:weight} holds with equality.
		Assume $1\le i\le r-2$ and set $N=T_{i+2}(M)$.
		Since $\overline{\alpha_N}\in\mathcal D_i$ and $v^\#(a)-v(a)\in\mathcal E_i$, we have
		\[
		\overline{B}_i(\overline{\alpha_N},v^\#(a)-v(a))=0.
		\]
		By the definition of $v(a)$ and Proposition~\ref{prop:intersection},
		\begin{equation}\label{eq:weighted-alpha}
			\overline{B}_i(\overline{\alpha_N},v^\#(a))
			=\overline{B}_i(\overline{\alpha_N},v(a))=\sum_{F\in\F_i(M)}
			\frac{a_F}{\nu_{i-1}(M|F)}\,B_i(\alpha_N,x_F)=\sum_{F\in\F_i(M)}a_F.
		\end{equation}
		By Proposition~\ref{prop:intersection}, the classes $\overline{x_F}$ with $F\in\F_i(M)$ are pairwise orthogonal.
		Hence
		\begin{equation}\label{eq:weighted-square}
			\overline{B}_i(v(a),v(a))
			=\sum_{F\in\F_i(M)}
			\frac{a_F^2}{\nu_{i-1}(M|F)^2}\,
			\overline{B}_i(\overline{x_F},\overline{x_F})=-\sum_{F\in\F_i(M)}
			\frac{|\F_1(M/F)|-1}{\nu_{i-1}(M|F)}\,a_F^2.
		\end{equation}
		Since $v^\#(a)\in\mathcal D_i$ and $v^\#(a)-v(a)\in\mathcal E_i$,
		we have
		\[
		\overline{B}_i\bigl(v^\#(a),v^\#(a)-v(a)\bigr)=0.
		\]
		Using~\eqref{eq:chow-numerical-pullback} and~\eqref{eq:weighted-cross}, we obtain
		\begin{equation}\label{eq:chow-pullback-square}
			\overline{B}_i(v^\#(a),v^\#(a))
			=\overline{B}_i(v(a),v^\#(a))=\overline{B}_i(v(a),v(a))
			+\sum_{H\in\F_{i+1}(M)}
			\frac{\left(\displaystyle\sum_{F\in\F_i(M|H)}a_F\right)^2}
			{\nu_{i-1}(M|H)}.
		\end{equation}
		
		By~\eqref{eq:chow-alpha-square}, $\overline{B}_i(\overline{\alpha_N},\overline{\alpha_N})=\nu_{i-1}(M)>0$.
		Set
		\[
		u=v^\#(a)-
		\frac{\overline{B}_i(\overline{\alpha_N},v^\#(a))}
		{\overline{B}_i(\overline{\alpha_N},\overline{\alpha_N})}
		\,\overline{\alpha_N}\in\mathcal D_i.
		\]
		Then $\overline{B}_i(\overline{\alpha_N},u)=0$.
		By Proposition~\ref{prop:chow-quotient-pairing}, the form $\overline{B}_i$ has exactly one positive eigenvalue.
		If $\overline{B}_i(u,u)>0$, then the orthogonal vectors $\overline{\alpha_N}$ and $u$ would span a two-dimensional subspace on which $\overline{B}_i$ is positive definite.
		Hence $\overline{B}_i(u,u)\le 0$.
		Thus
		\begin{equation}\label{eq:chow-primitive-square}
			\overline{B}_i(u,u)
			=
			\overline{B}_i(v^\#(a),v^\#(a))
			-
			\frac{\overline{B}_i(\overline{\alpha_N},v^\#(a))^2}
			{\overline{B}_i(\overline{\alpha_N},\overline{\alpha_N})}
			\le 0.
		\end{equation}
		
		Substituting \eqref{eq:chow-alpha-square}, \eqref{eq:weighted-alpha}, \eqref{eq:weighted-square}, and \eqref{eq:chow-pullback-square} into \eqref{eq:chow-primitive-square}, we obtain
		\[
		\sum_{H\in\F_{i+1}(M)}
		\frac{\left(\displaystyle\sum_{F\in\F_i(M|H)}a_F\right)^2}
		{\nu_{i-1}(M|H)}
		\le
		\frac{\left(\displaystyle\sum_{F\in\F_i(M)}a_F\right)^2}
		{\nu_{i-1}(M)}
		+
		\sum_{F\in\F_i(M)}
		\frac{|\F_1(M/F)|-1}{\nu_{i-1}(M|F)}\,a_F^2.
		\]
		This completes the proof.
	\end{proof}
	
	\begin{remark}\label{rem:schur}
		For $1\le i\le r-2$, the weighted inequality \eqref{ineq:weight} can be written as a Schur complement inequality.
		Write $\F_{i+1}(M)=\{H_1,\ldots,H_m\}$ and let $G$ be the Gram matrix of
		\[
		v(a),\ \overline{\alpha_N},\
		\overline{x_{H_1}},\ldots,\overline{x_{H_m}}
		\]
		with respect to $\overline{B}_i$.
		By~\eqref{eq:chow-alpha-square} and Proposition~\ref{prop:intersection}, deleting the first row and column of $G$ gives the invertible matrix
		\[
		P=\diag\bigl(\nu_{i-1}(M),
		-\nu_{i-1}(M|H_1),\ldots,-\nu_{i-1}(M|H_m)\bigr).
		\]
		The Schur complement of $P$ in $G$ is
		\[
		\overline{B}_i(v(a),v(a))
		-\frac{\overline{B}_i(v(a),\overline{\alpha_N})^2}{\nu_{i-1}(M)}
		+\sum_{H\in\F_{i+1}(M)}
		\frac{\overline{B}_i(v(a),\overline{x_H})^2}{\nu_{i-1}(M|H)}
		=\overline{B}_i(u,u),
		\]
		where $u$ is the class used in~\eqref{eq:chow-primitive-square}.
		By~\eqref{eq:weighted-cross}, \eqref{eq:weighted-alpha}, and~\eqref{eq:weighted-square}, this Schur complement is nonpositive precisely when~\eqref{ineq:weight} holds.
	\end{remark}
	
	\section{Applications}\label{sec:app}
	In this section, we apply Theorem~\ref{thm:weighted-incidence} to prove the points--lines--planes conjecture and Brylawski's conjecture on Whitney numbers of the first kind.
	We then recover Mason's ultra-log-concavity inequality for independent sets.
	
	\subsection{The points--lines--planes conjecture}\label{sec:plp}
	Let $M$ be a matroid of rank $r\ge3$ on $E$.
	Two distinct elements $e,f\in E$ are \emph{parallel} if
	$\rk_M(\{e\})=\rk_M(\{f\})=\rk_M(\{e,f\})=1$.
	A matroid is \emph{simple} if it has no loops and no distinct parallel elements.
	A \emph{parallel class} is a maximal set of nonloops in which every two distinct elements are parallel.
	Let $\widetilde E\subseteq E$ contain no loops and exactly one element from each parallel class.
	The restriction $\widetilde M:=M|\widetilde E$ is a \emph{simplification} of $M$.
	The matroid $\widetilde M$ is simple, and $\mathcal{L}(\widetilde M)\cong\mathcal{L}(M)$; see~\cite[Section~1.7]{Oxley11}.
	Thus $\widetilde M$ has rank $r$, and $W_i(\widetilde M)=W_i(M)$ for $0\le i\le r$.
	This isomorphism also preserves the number of points in each plane.
	We may therefore assume throughout this subsection that $M$ is simple.
	We use the following matroid formulation of the de Bruijn--Erd\H{o}s theorem.
	
	\begin{theorem}[{\cite{dBE48}}]
		\label{thm:dBE48}
		Let $M$ be a simple matroid of rank $r\ge3$.
		Then $W_2(M)\ge W_1(M)$.
	\end{theorem}
	
	We need two inequalities to prove Theorem~\ref{thm:splp}.
	The first bounds
	$\sum_{L\in\F_2(M)}1/(|L|-1)$
	from below in terms of the numbers of points and lines.
	
	\begin{lemma}\label{lem:pair}
		Let $M$ be a simple matroid of rank at least $3$.
		Then
		\begin{equation}\label{eq:pair}
			\sum_{L\in\F_2(M)}\frac1{|L|-1}
			\ge
			\frac{W_2(M)^2}{\binom{W_1(M)}2}.
		\end{equation}
		Equality holds if and only if every line contains exactly
		two points.
	\end{lemma}
	
	\begin{proof}
		Since $M$ is simple, we have $W_1(M)=|E|$.
		Every pair of distinct elements is contained in a unique line.
		Thus
		\begin{equation*}
			\sum_{L\in\F_2(M)}\binom{|L|}{2}
			=\binom{W_1(M)}2.
		\end{equation*}
		For every $L\in\F_2(M)$, we have $|L|\ge2$, so $\binom{|L|}{2}-(|L|-1)=\frac{(|L|-1)(|L|-2)}2\ge0$.
		Therefore
		\begin{equation}\label{l-1}
			\sum_{L\in\F_2(M)}(|L|-1)
			\le
			\sum_{L\in\F_2(M)}\binom{|L|}{2}
			=
			\binom{W_1(M)}2,
		\end{equation}
		with equality if and only if $|L|=2$ for every line $L$.
		
		The Cauchy--Schwarz inequality gives
		\begin{align*}
			W_2(M)^2&=\left(\sum_{L\in\F_2(M)}\sqrt{|L|-1}\,\frac1{\sqrt{|L|-1}}\right)^2\\
			&\le
			\left(\sum_{L\in\F_2(M)}\left(|L|-1\right)\right)
			\left(\sum_{L\in\F_2(M)}\frac1{|L|-1}\right)\\
			&\le
			\binom{W_1(M)}2
			\sum_{L\in\F_2(M)}\frac1{|L|-1},
		\end{align*}
		where the last inequality follows from~\eqref{l-1}.
		This proves \eqref{eq:pair}.
		Equality holds if and only if $|L|=2$ for every line $L$.
	\end{proof}
	
	The second inequality concerns simple matroids of rank $3$.
	
	\begin{lemma}\label{lem:rankthree}
		Let $Q$ be a simple matroid of rank $3$.
		Then
		\begin{equation}\label{eq:rankthree}
			\frac{W_2(Q)^2}{W_1(Q)-1}
			-
			\sum_{L\in\F_2(Q)}\frac1{|L|-1}
			\ge\frac32.
		\end{equation}
		Equality holds if and only if $W_1(Q)=3$.
	\end{lemma}
	
	\begin{proof}
		Since $Q$ has rank $3$, we have $W_1(Q)\ge3$.
		We consider two cases.
		
		\medskip
		\noindent
		\textbf{Case 1.} There exists $L_0\in\F_2(Q)$ with $|L_0|\ge3$.
		Then $\frac{1}{|L_0|-1}\le\frac12$.
		For every other line $L$, we have $\frac{1}{|L|-1}\le1$.
		Therefore
		\begin{equation*}
			\sum_{L\in\F_2(Q)}\frac1{|L|-1}
			=
			\frac1{|L_0|-1}
			+
			\sum_{\substack{L\in\F_2(Q)\\L\ne L_0}}
			\frac1{|L|-1}\le
			W_2(Q)-\frac12.
		\end{equation*}
		By Theorem~\ref{thm:dBE48}, we have $W_2(Q)\ge W_1(Q)$.
		Hence
		\begin{align*}
			\frac{W_2(Q)^2}{W_1(Q)-1}-\sum_{L\in\F_2(Q)}\frac1{|L|-1}
			&\ge \frac{W_2(Q)^2}{W_1(Q)-1}-W_2(Q)+\frac12\\
			&=\frac{W_2(Q)}{W_1(Q)-1}\,\left(W_2(Q)-W_1(Q)+1\right)+\frac12\\
			&\ge\frac1{W_1(Q)-1}+\frac32>\frac32.
		\end{align*}
		Thus~\eqref{eq:rankthree} is strict in this case.
		
		\medskip
		\noindent
		\textbf{Case 2.}
		Every line $L\in \F_2(Q)$ satisfies $|L|=2$.
		Then
		\begin{equation*}
			W_2(Q)=\binom{W_1(Q)}2,\qquad \sum_{L\in\F_2(Q)}\frac1{|L|-1}=W_2(Q).
		\end{equation*}
		Therefore
		\begin{equation*}
			\frac{W_2(Q)^2}{W_1(Q)-1}
			-
			\sum_{L\in\F_2(Q)}\frac1{|L|-1}
			=
			\frac{\binom{W_1(Q)}2^2}{W_1(Q)-1}
			-\binom{W_1(Q)}2=
			\frac32\binom{W_1(Q)}3
			\ge\frac32.
		\end{equation*}
		Equality holds in \eqref{eq:rankthree} if and only if $W_1(Q)=3$.
	\end{proof}
	
	\begin{proof}[Proof of Theorem~\ref{thm:splp}]
		By the preceding reduction, it suffices to consider a simple matroid $M$ of rank $r\ge3$.
		
		Taking $i=2$ and $a_L=1$ for every $L\in\F_2(M)$
		in~\eqref{ineq:weight},
		we obtain
		\begin{equation}\label{eq:thm1.5}
			\sum_{H\in\F_{3}(M)}
			\frac{\left(\sum_{L\in \F_{2}(M|H)} 1\right)^2}
			{\nu_{1}(M|H)}
			\le
			\frac{\left(\sum_{L\in\F_2(M)} 1\right)^2}
			{\nu_{1}(M)}
			+
			\sum_{L\in\F_2(M)}
			\frac{|\F_1(M/L)|-1}{\nu_{1}(M|L)}.
		\end{equation}
		For $L\in\F_2(M)$ and $H\in\F_3(M)$, the restrictions $M|L$ and $M|H$ are simple.
		By~\eqref{eq:w-nu},
		\begin{equation*}
			\nu_1(M)=W_1(M)-1,\qquad
			\nu_1(M|L)=|L|-1,\qquad
			\nu_1(M|H)=|H|-1.
		\end{equation*}
		Substituting these into~\eqref{eq:thm1.5} gives
		\begin{equation}\label{eq:plp-first}
			\sum_{H\in\F_3(M)}
			\frac{|\F_2(M|H)|^2}{|H|-1}
			\le
			\frac{W_2(M)^2}{W_1(M)-1}
			+
			\sum_{L\in\F_2(M)}
			\frac{|\F_1(M/L)|-1}{|L|-1}.
		\end{equation}
		
		For every $L\in\F_2(M)$, the rank-one flats of $M/L$
		correspond to the planes of $M$ containing $L$.
		Thus,
		\begin{equation*}
			\sum_{L\in\F_2(M)}
			\frac{|\F_1(M/L)|}{|L|-1}
			=
			\sum_{H\in\F_3(M)}
			\sum_{L\in\F_2(M|H)}\frac1{|L|-1}.
		\end{equation*}
		We can therefore rewrite~\eqref{eq:plp-first} as
		\begin{equation}\label{eq:plp-localization}
			\sum_{H\in\F_3(M)}
			\left(
			\frac{|\F_2(M|H)|^2}{|H|-1}
			-
			\sum_{L\in\F_2(M|H)}\frac1{|L|-1}
			\right)
			\le
			\frac{W_2(M)^2}{W_1(M)-1}
			-
			\sum_{L\in\F_2(M)}\frac1{|L|-1}.
		\end{equation}
		
		For every $H\in\F_3(M)$, the restriction $M|H$ is simple of rank $3$, with
		\begin{equation*}
			W_1(M|H)=|H|,
			\qquad
			W_2(M|H)=|\F_2(M|H)|.
		\end{equation*}
		By Lemma~\ref{lem:rankthree},
		\begin{equation*}
			\frac{|\F_2(M|H)|^2}{|H|-1}
			-
			\sum_{L\in\F_2(M|H)}\frac1{|L|-1}
			\ge\frac32.
		\end{equation*}
		Summing over $H\in\F_3(M)$ and using
		\eqref{eq:plp-localization},
		we obtain
		\begin{align}
			\frac32 W_3(M)&\le\sum_{H\in\F_3(M)}\left(\frac{|\F_2(M|H)|^2}{|H|-1}-\sum_{L\in\F_2(M|H)}\frac1{|L|-1}\right)
			\label{eq:plp-plane-bound}\\
			&\le\frac{W_2(M)^2}{W_1(M)-1}-
			\sum_{L\in\F_2(M)}\frac1{|L|-1}
			\nonumber\\
			&\le\frac{W_2(M)^2}{W_1(M)-1}-\frac{W_2(M)^2}{\binom{W_1(M)}2}
			\nonumber\\
			&=\frac{W_1(M)-2}{W_1(M)(W_1(M)-1)}W_2(M)^2,
			\nonumber
		\end{align}
		where the last inequality follows from Lemma~\ref{lem:pair}.
		Since $W_1(M)\ge3$, we have
		\begin{equation*}
			W_2(M)^2
			\ge
			\frac{3\,(W_1(M)-1)}{2\,(W_1(M)-2)}
			\,W_1(M)\,W_3(M).
		\end{equation*}
		This proves the required inequality.
		
		Suppose that equality holds in~\eqref{eq:splp}.
		Then equality holds in~\eqref{eq:plp-plane-bound}.
		By the equality criterion in Lemma~\ref{lem:rankthree}, every plane contains exactly three points.
		Conversely, suppose that every plane contains exactly three points.
		Every line is contained in a plane, so each line has exactly two points.
		It follows that every two-element subset is a line and every three-element subset is a plane.
		Therefore
		\begin{equation*}
			W_2(M)=\binom{W_1(M)}2,
			\qquad
			W_3(M)=\binom{W_1(M)}3.
		\end{equation*}
		Hence
		\begin{equation*}
			W_2(M)^2
			=\binom{W_1(M)}2^2
			=\frac{3(W_1(M)-1)}{2(W_1(M)-2)}
			W_1(M)W_3(M).
		\end{equation*}
		This proves the equality criterion.
	\end{proof}
	
	\subsection{Brylawski's conjecture}\label{sec:brylawski}
	Let $M$ be a loopless matroid of rank $r\ge1$ on $E$.
	A set $I\subseteq E$ is \emph{independent} if $\rk_M(I)=|I|$, and is \emph{dependent} otherwise.
	A \emph{circuit} is a minimal dependent set.
	Fix a total order on $E$, and denote its least element by $0$.
	A \emph{broken circuit} is a set $C\setminus\{\min C\}$, where $C$ is a circuit of $M$.
	The \emph{reduced broken-circuit complex} of $M$ is
	\begin{equation*}
		\rBC(M)
		=
		\{S\subseteq E\setminus\{0\}:
		S\text{ contains no broken circuit}\}.
	\end{equation*}
	This family is closed under taking subsets and hence is a simplicial complex.
	Its members are called \emph{faces}, and its maximal faces are called \emph{facets}.
	We refer the reader to~\cite{Brylawski77} for further background
	on broken-circuit complexes.
	
	We use the following form of the broken-circuit theorem.
	
	\begin{thm}[{\cite{Brylawski77}}]\label{thm:bc}
		Let $M$ be a loopless matroid of rank $r\ge1$, with a fixed total order on $E$.
		Every facet of $\rBC(M)$ has $r-1$ elements, and
		\begin{equation*}
			\nu_i(M)=
			|\{S\in\rBC(M):|S|=i\}|
			\quad\text{for all }0\le i\le r-1.
		\end{equation*}
	\end{thm}
	
	We next establish two inequalities for the reduced characteristic coefficients.
	
	\begin{lemma}\label{lem:facecounts}
		Let $M$ be a loopless matroid of rank $r\ge2$ on $n$ elements.
		Then
		\begin{equation}\label{eq:local-facecount}
			\nu_{r-2}(M)\le (r-1)\,\nu_{r-1}(M).
		\end{equation}
		Moreover, for every $1\le i\le r-2$, we have
		\begin{equation}\label{eq:global-facecount}
			(i+1)\,\nu_{i+1}(M)\le(n-i-1)\,\nu_i(M).
		\end{equation}
		Equality in~\eqref{eq:global-facecount} holds if and only if
		every circuit of $M$ has at least $i+3$ elements.
	\end{lemma}
	
	\begin{proof}
		Let $E$ be the ground set of $M$, with $|E|=n$.
		Fix a total order on $E$ with least element $0$.
		By Theorem~\ref{thm:bc}, every facet of $\rBC(M)$ has $r-1$ elements.
		Let
		\begin{equation*}
			\mathcal P=
			\{(S,T)\in\rBC(M)\times\rBC(M):
			S\subseteq T,\ |S|=r-2,\ |T|=r-1\}.
		\end{equation*}
		Each facet $T$ contains exactly $r-1$ faces of size $r-2$.
		Counting the pairs by $T$ and using Theorem~\ref{thm:bc}, we obtain
		\begin{equation}\label{Pequal}
			|\mathcal P|=(r-1)\,\nu_{r-1}(M).
		\end{equation}
		Every face of size $r-2$ lies in a facet, since all facets have size $r-1$.
		Hence $|\mathcal P|\ge\nu_{r-2}(M)$.
		Together with~\eqref{Pequal}, this proves~\eqref{eq:local-facecount}.
		
		We prove~\eqref{eq:global-facecount} by a similar count.
		Fix $1\le i\le r-2$.
		Let
		\begin{equation*}
			\mathcal P'=
			\{(S,T)\in\rBC(M)\times\rBC(M):
			S\subseteq T,\ |S|=i,\ |T|=i+1\}.
		\end{equation*}
		Each face $T$ of size $i+1$ contains exactly $i+1$ faces of size $i$, so
		\begin{equation}\label{Pequal2}
			|\mathcal P'|=(i+1)\,\nu_{i+1}(M).
		\end{equation}
		For a fixed face $S$ of size $i$, every such $T$ has the form
		$S\cup\{t\}$ with $t\in E\setminus(S\cup\{0\})$.
		There are $n-i-1$ choices for $t$, so
		$|\mathcal P'|\le(n-i-1)\,\nu_i(M)$.
		Combining this with~\eqref{Pequal2} proves~\eqref{eq:global-facecount}.
		
		Suppose that equality holds in \eqref{eq:global-facecount}.
		Then every face $S$ with $|S|=i$ is contained in exactly $n-i-1$ faces with $i+1$ elements.
		Thus
		\begin{equation}\label{eq:face-extension}
			S\cup\{t\}\in\rBC(M)
			\quad \text{for all }\,
			t\in E\setminus(S\cup\{0\}).
		\end{equation}
		
		Let $S\subseteq E\setminus\{0\}$ with $|S|=i$.
		We claim that $S$ is a face.
		Choose a facet $F$ of $\rBC(M)$ and an $i$-element subset $S'\subseteq F$.
		Such a subset exists because $|F|=r-1\ge i$, and it is a face.
		If $S'\ne S$, choose $s\in S'\setminus S$ and
		$t\in S\setminus S'$.
		By \eqref{eq:face-extension}, $S'\cup\{t\}$ is a face.
		Thus $(S'\setminus\{s\})\cup\{t\}$ is also a face.
		Replacing $S'$ by this face decreases $|S'\setminus S|$ by one.
		Repeating the argument shows that $S$ is a face.
		Hence every $i$-element subset of $E\setminus\{0\}$
		is a face.
		By \eqref{eq:face-extension}, every $(i+1)$-element
		subset is also a face.
		Thus, every subset of $E\setminus\{0\}$ with at most $i+1$
		elements is a face.
		For every circuit $C$ of $M$, we have
		\begin{equation*}
			C\setminus\{\min C\}\subseteq E\setminus\{0\},
			\qquad
			C\setminus\{\min C\}\notin\rBC(M).
		\end{equation*}
		Therefore $|C|-1\ge i+2$, or equivalently, $|C|\ge i+3$.
		
		Conversely, suppose that every circuit of $M$
		has at least $i+3$ elements.
		Then every broken circuit has at least $i+2$ elements,
		so every subset of $E\setminus\{0\}$ with at most $i+1$
		elements is a face.
		By Theorem~\ref{thm:bc}, we have $\nu_j(M)=\binom{n-1}{j}$ for $0\le j\le i+1$.
		Thus,
		\begin{equation*}
			(i+1)\,\nu_{i+1}(M)
			=(i+1)\binom{n-1}{i+1}=(n-i-1)\binom{n-1}{i}=(n-i-1)\,\nu_i(M).
		\end{equation*}
		Hence equality holds in \eqref{eq:global-facecount}.
	\end{proof}
	
	We apply Theorem~\ref{thm:weighted-incidence} with the weights $a_F=|\mu(\varnothing,F)|$ to obtain an inequality for the reduced characteristic coefficients.
	
	\begin{proposition}\label{prop:local-coefficients}
		Let $M$ be a loopless matroid of rank $r$, and let $1\le i\le r-2$.
		Then
		\begin{equation}\label{eq:local-coefficients}
			\frac{\nu_i(M)^2}{\nu_{i-1}(M)}
			\ge
			\nu_{i+1}(M)
			+
			\sum_{H\in\F_{i+1}(M)}
			\frac{\nu_i(M|H)^2}{\nu_{i-1}(M|H)}.
		\end{equation}
	\end{proposition}
	
	\begin{proof}
		For every $F\in\F_i(M)$, the restriction $M|F$ has rank $i$ and $\mathcal L(M|F)=[\varnothing,F]$.
		Its characteristic and reduced characteristic polynomials are
		\begin{equation*}
			\chi_{M|F}(t)=
			\sum_{\substack{G\in\mathcal L(M)\\G\subseteq F}}
			\mu(\varnothing,G)t^{i-\rk_M(G)},\qquad \overline{\chi}_{M|F}(t)
			=\sum_{j=0}^{i-1}
			(-1)^j\nu_j(M|F)t^{i-1-j}.
		\end{equation*}
		Their constant terms are $\chi_{M|F}(0)=\mu(\varnothing,F)$ and
		$\overline{\chi}_{M|F}(0)=(-1)^{i-1}\nu_{i-1}(M|F).$
		Since
		$\chi_{M|F}(t)=(t-1)\overline{\chi}_{M|F}(t)$,
		we obtain
		\begin{equation*}
			\mu(\varnothing,F)=\chi_{M|F}(0)=-\overline{\chi}_{M|F}(0)=(-1)^i\nu_{i-1}(M|F).
		\end{equation*}
		Taking the weights
		\begin{equation}\label{eq:mobius-weights}
			a_F=|\mu(\varnothing,F)|=\nu_{i-1}(M|F)
			\quad\text{for all }F\in\F_i(M)
		\end{equation}
		in~\eqref{ineq:weight}, we obtain
		\begin{equation}\label{weightcase2}
			\sum_{H\in\F_{i+1}(M)}
			\frac{\left(\sum_{F\in\F_i(M|H)}a_F\right)^2}
			{\nu_{i-1}(M|H)}
			\le
			\frac{\left(\sum_{F\in\F_i(M)}a_F\right)^2}
			{\nu_{i-1}(M)}
			+
			\sum_{F\in\F_i(M)}
			\bigl(|\F_1(M/F)|-1\bigr)\,a_F.
		\end{equation}
		
		Writing $\mu_M$ for the M\"obius function of $\mathcal L(M)$, we have
		\begin{align}\label{eq:mobius-sums}
			\sum_{F\in\F_i(M)}a_F
			&=\sum_{F\in\F_i(M)}|\mu_M(\varnothing,F)|=(-1)^i\sum_{F\in\F_i(M)}\mu_M(\varnothing,F)
			=w_i(M).
		\end{align}
		For every $H\in\F_{i+1}(M)$ and $F\in\F_i(M|H)$,
		restriction to $H$ leaves the interval $[\varnothing,F]$ unchanged.
		Writing $\mu_{M|H}$ for the M\"obius function of $\mathcal L(M|H)$, we obtain
		\begin{align}\label{eq:local-mobius-sums}
			\sum_{F\in\F_i(M|H)}a_F
			&=\sum_{F\in\F_i(M|H)}|\mu_M(\varnothing,F)| \notag\\
			&=\sum_{F\in\F_i(M|H)}|\mu_{M|H}(\varnothing,F)| \notag\\
			&=(-1)^i\sum_{F\in\F_i(M|H)}\mu_{M|H}(\varnothing,F)
			=w_i(M|H).
		\end{align}
		
		For every $F\in\F_i(M)$, the rank-one flats of $M/F$
		correspond to the flats $H\in\F_{i+1}(M)$ containing $F$.
		Counting the pairs $(F,H)$ with $F\in\F_i(M)$,
		$H\in\F_{i+1}(M)$, and $F\subseteq H$, each with weight $a_F$, gives
		\begin{align}\label{eq:mobius-double}
			\sum_{F\in\F_i(M)}
			\bigl(|\F_1(M/F)|-1\bigr)a_F
			&=\sum_{H\in\F_{i+1}(M)}
			\sum_{F\in\F_i(M|H)}a_F-\sum_{F\in\F_i(M)}a_F\nonumber\\
			&=\sum_{H\in\F_{i+1}(M)}w_i(M|H)-w_i(M).
		\end{align}
		Substituting~\eqref{eq:mobius-sums}, \eqref{eq:local-mobius-sums}, and
		\eqref{eq:mobius-double} into~\eqref{weightcase2} gives
		\begin{equation}\label{eq:mobius-incidence}
			\sum_{H\in\F_{i+1}(M)}
			\frac{w_i(M|H)^2}{\nu_{i-1}(M|H)}
			\le
			\frac{w_i(M)^2}{\nu_{i-1}(M)}
			+
			\sum_{H\in\F_{i+1}(M)}w_i(M|H)-w_i(M).
		\end{equation}
		
		By \eqref{eq:w-nu}, we have
		\begin{equation}\label{eq:mobius-w-nu}
			w_i(M)=\nu_i(M)+\nu_{i-1}(M),
			\qquad
			w_i(M|H)=\nu_i(M|H)+\nu_{i-1}(M|H).
		\end{equation}
		Substituting~\eqref{eq:mobius-w-nu} into
		\eqref{eq:mobius-incidence}, we obtain
		\begin{equation}\label{eq:5.3}
			\sum_{H\in\F_{i+1}(M)}
			\left(
			\frac{\nu_i(M|H)^2}{\nu_{i-1}(M|H)}
			+\nu_i(M|H)
			\right)
			\le
			\frac{\nu_i(M)^2}{\nu_{i-1}(M)}+\nu_i(M).
		\end{equation}
		
		For $H\in\F_{i+1}(M)$, the restriction $M|H$ has rank $i+1$, so
		$\nu_i(M|H)=|\mu(\varnothing,H)|$.
		Therefore
		\begin{equation}\label{eq:top-local-sum}
			\sum_{H\in\F_{i+1}(M)}\nu_i(M|H)
			=
			w_{i+1}(M)
			=
			\nu_i(M)+\nu_{i+1}(M),
		\end{equation}
		where the last equality follows from \eqref{eq:w-nu}.
		Combining this identity with~\eqref{eq:5.3}, we conclude that
		\begin{equation*}
			\frac{\nu_i(M)^2}{\nu_{i-1}(M)}
			\ge
			\nu_{i+1}(M)
			+
			\sum_{H\in\F_{i+1}(M)}
			\frac{\nu_i(M|H)^2}{\nu_{i-1}(M|H)}.
		\end{equation*}
		This proves \eqref{eq:local-coefficients}.
	\end{proof}
	
	We now deduce ultra-log-concavity for the reduced characteristic coefficients.
	
	\begin{thm}\label{thm:reduced-ulc}
		Let $M$ be a loopless matroid of rank $r\ge1$ on $n$ elements.
		For $1\le i\le r-2$, we have
		\begin{equation}\label{eq:reduced-ulc}
			\nu_i(M)^2
			\ge
			\frac{(i+1)(n-i)}{i(n-i-1)}
			\nu_{i-1}(M)\nu_{i+1}(M).
		\end{equation}
		Equality holds if and only if every circuit of $M$ has at least $i+3$ elements.
	\end{thm}
	
	\begin{proof}
		Fix $1\le i\le r-2$.
		For every $H\in\F_{i+1}(M)$, the restriction $M|H$ is loopless of rank $i+1$.
		Applying~\eqref{eq:local-facecount} to $M|H$, we obtain
		\begin{equation}\label{eq:local-ratio-bound}
			\nu_{i-1}(M|H)\le i\,\nu_i(M|H).
		\end{equation}
		By \eqref{eq:local-coefficients} and \eqref{eq:local-ratio-bound}, we have
		\begin{align}\label{eq:reduced-intermediate}
			\frac{\nu_i(M)^2}{\nu_{i-1}(M)}
			&\ge
			\nu_{i+1}(M)
			+
			\sum_{H\in\F_{i+1}(M)}
			\frac{\nu_i(M|H)^2}{\nu_{i-1}(M|H)}\nonumber\\
			&\ge
			\nu_{i+1}(M)
			+\frac1i\sum_{H\in\F_{i+1}(M)}\nu_i(M|H)\nonumber\\
			&=
			\frac{i+1}{i}\nu_{i+1}(M)+\frac1i\nu_i(M),
		\end{align}
		where the last equality follows from \eqref{eq:top-local-sum}.
		
		Since $n\ge r\ge i+2$, we have $n-i-1>0$.
		By \eqref{eq:global-facecount},
		\begin{equation}\label{eq:global-ratio-bound}
			\nu_i(M)\ge
			\frac{i+1}{n-i-1}\nu_{i+1}(M).
		\end{equation}
		Substituting \eqref{eq:global-ratio-bound} into~\eqref{eq:reduced-intermediate},
		we obtain
		\begin{equation*}
			\frac{\nu_i(M)^2}{\nu_{i-1}(M)}
			\ge
			\frac{i+1}{i}\nu_{i+1}(M)
			+\frac{i+1}{i(n-i-1)}\nu_{i+1}(M)=
			\frac{(i+1)(n-i)}{i(n-i-1)}\nu_{i+1}(M).
		\end{equation*}
		Since $\nu_{i-1}(M)>0$, this proves~\eqref{eq:reduced-ulc}.
		
		Suppose that equality holds in \eqref{eq:reduced-ulc}.
		Then equality holds in \eqref{eq:global-ratio-bound}.
		By Lemma~\ref{lem:facecounts}, every circuit of $M$ has at least $i+3$ elements.
		
		Conversely, suppose that every circuit of $M$ has at least $i+3$ elements.
		By the definition of $\rBC(M)$ and Theorem~\ref{thm:bc}, $\nu_j(M)=\binom{n-1}{j}$ for $j=i-1,i,i+1$.
		Hence equality holds in \eqref{eq:reduced-ulc}.
	\end{proof}
	
	Let $M_1$ and $M_2$ be matroids on disjoint ground sets $E_1$ and $E_2$.
	The \emph{direct sum} $M_1\oplus M_2$ is the matroid on $E_1\cup E_2$ with rank function
	\begin{equation*}
		\rk_{M_1\oplus M_2}(S)
		=
		\rk_{M_1}(S\cap E_1)+\rk_{M_2}(S\cap E_2)
		\quad \text{for all }S\subseteq E_1\cup E_2.
	\end{equation*}
	The circuits of $M_1\oplus M_2$ are precisely the circuits of its two summands.
	We have $\mathcal{L}(M_1\oplus M_2)\cong \mathcal{L}(M_1)\times\mathcal{L}(M_2)$, where the Cartesian product is ordered by $(F_1,F_2)\le(G_1,G_2)$ if and only if $F_1\subseteq G_1$ and $F_2\subseteq G_2$.
	The product formula for M\"obius functions~\cite[Proposition~3.8.2]{Stanley12} therefore gives
	\begin{equation}\label{eq:direct-sum-polynomial}
		\chi_{M_1\oplus M_2}(t)
		=
		\chi_{M_1}(t)\chi_{M_2}(t)
	\end{equation}
	for loopless $M_1$ and $M_2$.
	
	\begin{proof}[Proof of Theorem~\ref{thm:brylawski}]
		Let $M$ be a loopless matroid of rank $r\ge1$ on a ground set $E$ with $|E|=n$.
		Choose $e\notin E$
		and let $U$ be the matroid on $\{e\}$ of rank $1$.
		Then $U$ has no circuits and $\chi_{U}(t)=t-1$.
		Set $M'=M\oplus U$.
		Then $M'$ is loopless of rank $r+1$ on $n+1$ elements.
		
		By \eqref{eq:direct-sum-polynomial},
		\begin{equation*}
			\chi_{M'}(t)
			=
			\chi_M(t)\chi_{U}(t)
			=
			(t-1)\chi_M(t).
		\end{equation*}
		Thus $\overline{\chi}_{M'}(t)=\chi_M(t)$.
		Comparing coefficients, we obtain
		\begin{equation}\label{eq:coloop-coefficients}
			\nu_i(M')=w_i(M)
			\quad\text{for all } 0\le i\le r.
		\end{equation}
		
		Fix $1\le i\le r-1$.
		Applying Theorem~\ref{thm:reduced-ulc} to $M'$, we have
		\begin{equation*}
			\nu_i(M')^2
			\ge
			\frac{(i+1)(n-i+1)}{i(n-i)}
			\nu_{i-1}(M')\nu_{i+1}(M').
		\end{equation*}
		Substituting \eqref{eq:coloop-coefficients}, we obtain
		\begin{equation}\label{eq:full-ulc}
			w_i(M)^2
			\ge
			\frac{(i+1)(n-i+1)}{i(n-i)}
			w_{i-1}(M)w_{i+1}(M).
		\end{equation}
		
		Since $U$ has no circuits, the circuits of $M'$ are precisely the circuits of $M$.
		By Theorem~\ref{thm:reduced-ulc},
		equality holds in \eqref{eq:full-ulc} if and only if every circuit of $M$ has at least $i+3$ elements.
	\end{proof}
	
	\subsection{The strong Mason conjecture}\label{sec:mason}
	Let $M$ be a matroid of rank $r$ on a ground set $E$ with $|E|=n$.
	For $0\le j\le r$, let $I_j(M)$ denote the number of independent sets of size $j$.
	
	Welsh~\cite{Welsh71} conjectured that the sequence $(I_j(M))_{j=0}^r$ is unimodal.
	Mason~\cite{Mason72} strengthened this conjecture by proposing that the sequence is ultra-log-concave of order $n$.
	This conjecture, known as the \emph{strong Mason conjecture}, was independently proved by Br\"and\'en and Huh~\cite{BH20} and by Anari, Liu, Oveis Gharan, and Vinzant~\cite{ALGV24}.
	Murai, Nagaoka, and Yazawa~\cite[Corollary~1.2(ii)]{MNY21} characterized the cases of equality.
	
	Let $\mathcal I(M)$ denote the family of independent sets of $M$.
	The \emph{dual matroid} $M^*$ is the matroid on $E$ whose independent sets are
	\[
	\mathcal I(M^*)
	=
	\{A\subseteq E:\rk_M(E\setminus A)=r\}.
	\]
	For $e\notin E$, the \emph{free extension} of $M$ by $e$
	is the matroid $M+e$ on $E\cup\{e\}$ whose independent sets are
	\[
	\mathcal I(M+e)
	=
	\mathcal I(M)\cup
	\{I\cup\{e\}:I\in\mathcal I(M),\ |I|\le r-1\}.
	\]
	The \emph{free coextension} of $M$ by $e$ is defined by
	\[
	M\times e:=(M^*+e)^*;
	\]
	see~\cite{Lenz13}.
	
	We use the following identity of Brylawski~\cite{Brylawski77};
	see also~\cite[Remark~6.15.3(c)]{Brylawski82} and~\cite[Proposition~5]{Lenz13}.
	
	\begin{lemma}\label{lem:brylawski-identity}
		Let $M$ be a matroid of rank $r$ on $E$, and choose $e\notin E$.
		Then
		\[
		\overline{\chi}_{M\times e}(t)
		=\sum_{j=0}^{r}(-1)^j I_j(M)t^{r-j}.
		\]
		In particular,
		\begin{equation}\label{eq:coextension-coefficients}
			\nu_j(M\times e)=I_j(M)
			\quad\text{for all }0\le j\le r.
		\end{equation}
	\end{lemma}
	
	Theorem~\ref{thm:reduced-ulc} now gives the strong Mason inequality.
	
	\begin{corollary}\label{cor:strong-mason}
		Let $M$ be a matroid of rank $r\ge2$ on a ground set $E$ with $|E|=n$.
		For $1\le i\le r-1$, we have
		\begin{equation}\label{eq:strong-mason}
			I_i(M)^2
			\ge
			\frac{(i+1)(n-i+1)}{i(n-i)}
			I_{i-1}(M)I_{i+1}(M).
		\end{equation}
	\end{corollary}
	
	\begin{proof}
		Fix $1\le i\le r-1$, and let $N=M\times e$ for some $e\notin E$.
		The matroid $N$ is loopless of rank $r+1$ on $n+1$ elements.
		By Theorem~\ref{thm:reduced-ulc} and~\eqref{eq:coextension-coefficients}, we have
		\[
		I_i(M)^2
		=\nu_i(N)^2\ge
		\frac{(i+1)(n-i+1)}{i(n-i)}
		\nu_{i-1}(N)\nu_{i+1}(N)=
		\frac{(i+1)(n-i+1)}{i(n-i)}
		I_{i-1}(M)I_{i+1}(M).
		\]
		This proves~\eqref{eq:strong-mason}.
	\end{proof}
	
	\section*{Declaration of AI usage}
	During the preparation of this manuscript, the authors used generative AI tools to assist in exploring possible approaches, checking technical details, and improving the exposition. The authors verified the mathematical arguments, computations, proofs, and results and take full responsibility for the content of this manuscript.


\begin{thebibliography}{99}
		
		\bibitem{AHK18}
		K.~Adiprasito, J.~Huh, and E.~Katz,
		\emph{Hodge theory for combinatorial geometries},
		Ann. of Math. (2) \textbf{188} (2018), no.~2, 381--452.
		
		\bibitem{ALGV24}
		N.~Anari, K.~Liu, S.~Oveis Gharan, and C.~Vinzant,
		\emph{Log-concave polynomials III: Mason's ultra-log-concavity conjecture for independent sets of matroids},
		Proc. Amer. Math. Soc. \textbf{152} (2024), no.~5, 1969--1981.
		
		\bibitem{ADH23}
		F.~Ardila, G.~Denham, and J.~Huh,
		\emph{Lagrangian geometry of matroids},
		J. Amer. Math. Soc. \textbf{36} (2023), no.~3, 727--794.
		
		\bibitem{BHMPW22}
		T.~Braden, J.~Huh, J.~P.~Matherne, N.~Proudfoot, and B.~Wang,
		\emph{A semi-small decomposition of the Chow ring of a matroid},
		Adv. Math. \textbf{409} (2022), Paper No.~108646.
		
		\bibitem{BH20}
		P.~Br\"and\'en and J.~Huh,
		\emph{Lorentzian polynomials},
		Ann. of Math. (2) \textbf{192} (2020), no.~3, 821--891.
		
		\bibitem{BL26}
		P.~Br\"and\'en and J.~Leake,
		\emph{Lorentzian polynomials on cones},
		Forum Math. Sigma \textbf{14} (2026), Paper No.~e16, 34~pp.
		
		\bibitem{Brylawski77}
		T.~Brylawski,
		\emph{The broken-circuit complex},
		Trans. Amer. Math. Soc. \textbf{234} (1977), no.~2, 417--433.
		
		\bibitem{Brylawski82}
		T.~Brylawski,
		\emph{The Tutte polynomial. Part I: General theory},
		in \emph{Matroid Theory and Its Applications}, A.~Barlotti (ed.),
		Liguori, Naples, 1982, 125--275.
		
		\bibitem{dBE48}
		N.~G.~de Bruijn and P.~Erd\H{o}s,
		\emph{On a combinatorial problem},
		Proc. Kon. Nederl. Akad. Wetensch. \textbf{51} (1948), no.~10, 1277--1279.
		
		\bibitem{DCP95}
		C.~De Concini and C.~Procesi,
		\emph{Wonderful models of subspace arrangements},
		Selecta Math. (N.S.) \textbf{1} (1995), no.~3, 459--494.
		
		\bibitem{DLLW26}
		A.~Divoux, M.~Larson, C.~Lowen, and S.~Wang,
		\emph{Matroid flat counts can have many peaks},
		preprint, 2026.
		
		\bibitem{Duk04}
		M. Dukes,
		\emph{On the number of matroids on a finite set},
		S\'em. Lothar. Combin. \textbf{51} (2004/05), Art. B51g, 12 pp.
		
		\bibitem{Duk08}
		M. Dukes,
		\emph{Concerning the shape of a geometric lattice},
		Discrete Math. \textbf{308} (2008), no.~24, 6632--6638.
		
		\bibitem{FY04}
		E.~M.~Feichtner and S.~Yuzvinsky,
		\emph{Chow rings of toric varieties defined by atomic lattices},
		Invent. Math. \textbf{155} (2004), no.~3, 515--536.
		
		\bibitem{Hartshorne77}
		R.~Hartshorne,
		\emph{Algebraic Geometry},
		Graduate Texts in Mathematics, vol.~52, Springer-Verlag, New York--Heidelberg, 1977.
		
		\bibitem{Her72}
		A.~P. Heron,
		\emph{Matroid polynomials},
		in \emph{Combinatorics}
		(Proc. Conf. Combinatorial Math., Oxford, 1972),
		Institute of Mathematics and its Applications,
		Southend-on-Sea, 1972, 164--202.
		
		\bibitem{HJ13}
		R.~A.~Horn and C.~R.~Johnson,
		\emph{Matrix Analysis},
		second edition, Cambridge University Press, Cambridge, 2013.
		
		\bibitem{Huh12}
		J.~Huh,
		\emph{Milnor numbers of projective hypersurfaces and the chromatic polynomial of graphs},
		J. Amer. Math. Soc. \textbf{25} (2012), no.~3, 907--927.
		
		\bibitem{HK12}
		J.~Huh and E.~Katz,
		\emph{Log-concavity of characteristic polynomials and the Bergman fan of matroids},
		Math. Ann. \textbf{354} (2012), no.~3, 1103--1116.
		
		\bibitem{Jou83}
		J.-P.~Jouanolou,
		\emph{Th\'eor\`emes de Bertini et applications},
		Progress in Mathematics, vol.~42, Birkh\"auser Boston, Boston, MA, 1983.
		
		\bibitem{Kun00}
		J.~P.~S. Kung,
		\emph{On the lines--planes inequality for matroids},
		J. Combin. Theory Ser. A \textbf{91} (2000), no.~1-2, 363--368.
		
		\bibitem{Lenz13}
		M.~Lenz,
		\emph{The $f$-vector of a representable-matroid complex is log-concave},
		Adv. in Appl. Math. \textbf{51} (2013), no.~5, 543--545.
		
		\bibitem{Mason72}
		J.~H.~Mason,
		\emph{Matroids: unimodal conjectures and Motzkin's theorem},
		in \emph{Combinatorics} (Proc. Conf. Combinatorial Math., Oxford, 1972),
		D.~J.~A.~Welsh and D.~R.~Woodall (eds.),
		Institute of Mathematics and its Applications, Southend-on-Sea, 1972, 207--220.
		
		\bibitem{MMIB12}
		Y.~Matsumoto, S.~Moriyama, H.~Imai, and D.~Bremner,
		\emph{Matroid enumeration for incidence geometry},
		Discrete Comput. Geom. \textbf{47} (2012), no.~1, 17--43.
		
		\bibitem{Mumford61}
		D.~Mumford,
		\emph{The topology of normal singularities of an algebraic surface and a criterion for simplicity},
		Publ. Math. Inst. Hautes \'Etudes Sci. \textbf{9} (1961), 5--22.
		
		\bibitem{MNY21}
		S.~Murai, T.~Nagaoka, and A.~Yazawa,
		\emph{Strictness of the log-concavity of generating polynomials
			of matroids},
		J. Combin. Theory Ser. A \textbf{181} (2021), Paper No.~105351.
		
		\bibitem{Oxley11}
		J.~Oxley,
		\emph{Matroid Theory},
		second edition, Oxford Graduate Texts in Mathematics, vol.~21,
		Oxford University Press, Oxford, 2011.
		
		\bibitem{Pur86}
		G.~B. Purdy,
		\emph{Two results about points, lines and planes},
		Discrete Math. \textbf{60} (1986), 215--218.
		
		\bibitem{Rota71}
		G.-C.~Rota,
		\emph{Combinatorial theory, old and new},
		in \emph{Actes du Congr\`es International des Math\'ematiciens} (Nice, 1970),
		vol.~3, Gauthier-Villars, Paris, 1971, 229--233.
		
		\bibitem{Seymour82}
		P.~D.~Seymour,
		\emph{On the points-lines-planes conjecture},
		J. Combin. Theory Ser. B \textbf{33} (1982), no.~1, 17--26.
		
		\bibitem{Sta72}
		R.~P.~Stanley,
		\emph{Supersolvable lattices},
		Algebra Universalis \textbf{2} (1972), 197--217.
		
		\bibitem{Stanley12}
		R.~P.~Stanley,
		\emph{Enumerative Combinatorics. Volume~1},
		second edition, Cambridge Studies in Advanced Mathematics, vol.~49,
		Cambridge University Press, Cambridge, 2012.
		
		\bibitem{Stonesifer75}
		J.~R.~Stonesifer,
		\emph{Logarithmic concavity for edge lattices of graphs},
		J. Combin. Theory Ser. A \textbf{18} (1975), no.~1, 36--46.
		
		\bibitem{Welsh71}
		D.~J.~A.~Welsh,
		\emph{Combinatorial problems in matroid theory},
		in \emph{Combinatorial Mathematics and Its Applications} (Oxford, 1969),
		D.~J.~A.~Welsh (ed.), Academic Press, London--New York, 1971, 291--306.
		
		\bibitem{Welsh76}
		D.~J.~A.~Welsh,
		\emph{Matroid Theory},
		London Mathematical Society Monographs, vol.~8,
		Academic Press, London--New York, 1976.
		
	\end{thebibliography}
\end{document}